\documentclass[12pt,a4paper,reqno]{amsart} 
\usepackage{fullpage}
\usepackage{amsmath, amsthm, amssymb}
\usepackage{amscd,stmaryrd}
\usepackage[all,cmtip]{xy}

\usepackage{etoolbox}
\usepackage{soul}

\newtheorem{thm}{Theorem}[section]
\newtheorem{theorem}[thm]{Theorem}
\newtheorem{corollary}[thm]{Corollary}
\newtheorem{lemma}[thm]{Lemma}
\newtheorem{proposition}[thm]{Proposition}

\theoremstyle{definition}
\newtheorem{definition}[thm]{Definition}
\newtheorem{example}[thm]{Example}
\AtEndEnvironment{example}{\null\hfill \tiny{$\blacksquare$}}

\newtheorem{observation}[thm]{Observation}

\newtheorem{remark}[thm]{Remark}

\newcommand{\nn}{\nonumber}

\newcommand{\C}{\mathbb{C}}
\newcommand{\N}{\mathbb{N}}

\newcommand{\bP}{\mathbf{P}}

\newcommand{\cI}{\mathcal{I}}
\newcommand{\cO}{\mathcal{O}}
\newcommand{\cF}{\mathcal{F}}
\newcommand{\T}{a}
\newcommand{\co}{\mathrm{co}}

\newcommand{\rGL}{\mathrm{GL}}
\newcommand{\rSL}{\mathrm{SL}}
\newcommand{\rO}{\mathrm{O}}
\newcommand{\rSp}{\mathrm{Sp}}

\newcommand{\Lie}{\mathrm{Lie}}

\newcommand{\lra}{\longrightarrow}

\newcommand{\beq}{\begin{equation}}
\newcommand{\eeq}{\end{equation}}

\newcommand{\Ubar}{\overline{U}}

\newcommand{\gu}[1]{u^{\scriptscriptstyle{(#1)}}}
\newcommand{\gn}[1]{n^{\scriptscriptstyle{(#1)}}}
\newcommand{\gb}[1]{b^{\scriptscriptstyle{(#1)}}}

\newcommand{\zero}[1]{{#1}_{\scriptscriptstyle{(0)}}}
\newcommand{\one}[1]{{#1}_{\scriptscriptstyle{(1)}}}
\newcommand{\two}[1]{{#1}_{\scriptscriptstyle{(2)}}}
\newcommand{\three}[1]{{#1}_{\scriptscriptstyle{(3)}}}
\newcommand{\tone}[1]{{#1}^{\scriptscriptstyle{<1>}}}
\newcommand{\ttwo}[1]{{#1}^{\scriptscriptstyle{<2>}}}

\newcommand{\tuno}[1]{{#1}^{\scriptscriptstyle{<1>}}}
\newcommand{\tdue}[1]{{#1}^{\scriptscriptstyle{<2>}}}

\title[Title] 
{Reduction of Quantum Principal Bundles \\~\\ over non affine bases} 

\date{March 2024}

\author{Rita Fioresi, Emanuele Latini, Chiara Pagani}

\address[]{\textit{Rita Fioresi}  
\newline \indent    Alma Mater Studiorum Universit\`a di Bologna,
\newline \indent FaBiT,  via S. Donato 15, 40126 Bologna, Italy
\newline \indent Istituto Nazionale di Fisica Nucleare, Sezione di Bologna,  40126 Bologna, Italy.}
\email{rita.fioresi@unibo.it}
\address[]{\textit{Emanuele Latini} 
\newline \indent    Alma Mater Studiorum Universit\`a di Bologna,
\newline \indent Dipartimento di
Matematica,  Piazza di Porta S. Donato, 5, 40126 Bologna, Italy
\newline \indent Istituto Nazionale di Fisica Nucleare, Sezione di Bologna,  40126 Bologna, Italy.}
\email{emanuele.latini@unibo.it}
\address[]{\textit{Chiara Pagani} 
\newline \indent    Alma Mater Studiorum Universit\`a di Bologna,
\newline \indent Dipartimento di
Matematica,  Piazza di Porta S. Donato, 5, 40126 Bologna, Italy.}
\email{c.pagani@unibo.it}
\begin{document}

\begin{abstract}
In this paper we develop the theory of reduction of 
quantum principal bundles over   projective bases. We show how the
sheaf theoretic approach can be
effectively applied to certain relevant examples as the Klein model for the projective spaces; in particular we study in the algebraic setting the reduction of the principal bundle  $\rGL(n) \to \rGL(n)/P= \bP^{n-1}(\C)$ to  the Levi subgroup $G_0$ inside the maximal parabolic subgroup $P $ of  $\rGL(n)$.  
We characterize reductions in the sheaf theoretic setting.
\end{abstract}
\maketitle

\section{Introduction} \label{intro-sec}

A reduction of a principal $P$--bundle $E\to M$ to the structure
group $K$, with $K\subset P$ a subgroup, is the data of a principal
$K$ bundle $E_0\to M$ and a bundle morphism
$\phi:E_0\to E$.
Many relevant geometrical structures arise in this way.
Consider for example  a $K$ structure, that is the reduction to $K\subset \rGL_n(\mathbb{R})$ of the principal
$\rGL_n(\mathbb{R})$ frame bundle
$\mathcal{P}M\to M$. When $K=\rO_n(\mathbb{R})$ one gets a Riemannian structure on $M$, while when $n$ is even, and $K=\rSp_n(\mathbb{R})$
the reduction of the frame bundle induces an almost
symplectic structure on $M$.
In the context of parabolic geometries (see for example \cite{cap}), one considers a graded Lie
algebra $\mathfrak{g}=\mathfrak{n}_-\oplus \mathfrak{g}_0\oplus \mathfrak{n}_+$
and a principal $P$--bundle $\mathcal{G}\to M$ equipped with a
Cartan connection valued in the Lie algebra $\mathfrak{g}$,
where $\Lie(P)=\mathfrak{p}:=\mathfrak{g}_0\oplus \mathfrak{n}_+$; the image under the exponential map of $\mathfrak{n}_+$
is a closed nilpotent subgroup of $P$, that we denote by $N$, and
$P/N\cong G_0$ is the Levi subgroup of $P$, so that we have $P=G_0\ltimes N$.
It is interesting to observe that many information on the underlying geometrical structure are contained
in the reduced $G_0$--bundle $\mathcal{G}/N\to M$. Viewing this bundle  as a reduction of the original one
requires making a choice, namely an equivariant section of the projection $\mathcal{G}\to \mathcal{G}/N$.
This is exactly the choice of a Weyl structure and it turns out that any reduction of the
original bundle to the structure group $G_0$ is obtained in this way (see \cite[Proposition 5.1.1]{cap}).
When the total space is a Lie group $G$ `exponentiating'
the Lie algebra $\mathfrak{g}$, one obtains  the Klein description of 
flag varieties. Inspired by that, in  this manuscript we focus on the quantization of those structures. 

Starting from the pioneering work on Hopf-Galois extensions
 \cite{sch90} and quantum principal bundles  \cite{brz-maj}, quantum homogeneous spaces have been an important object of study
in the past years. Stimulated by the canonical non commutative differential calculi developed in \cite{HK1, HK2}, (see also \cite{bm} and refs. therein for a more general discussion on differential calculi on comodule algebras)  there was a recent revived interest on quantum flag varieties (see e.g. \cite{buachalla,Car-Bua}). It is important to observe that the classical framework detailed above, namely the one inspired by {\sl parabolic geometries},  cannot be fully
understood and treated by looking at affine objects only.  For example given a complex algebraic group $G$ and a parabolic subgroup $P$ as before, the model geometry $G/P$ 
is projective; when looking at the principal $P$-bundle $G \lra G/P$,
affine algebras are only a local model and cannot be used for the global
Hopf-Galois description, as in the above mentioned works. This problem was faced and solved only recently in \cite{afl} via the  use of sheaf theoretic 
methods,   introduced in \cite{pflaum, cp}, and non commutative algebraic geometric techniques
that were already conceptually introduced much earlier, in \cite{vov, ro}. Those techniques also proved essential to build
non commutative differential calculi on
quantum principal bundles \cite{AFLW}  and more specifically to treat quantum flag manifolds \cite{fi1, fi4, cfg}.

In this paper we  study reductions of quantum principal
bundles in the sheaf theoretic language. We will start from the affine
case, introduced in \cite{gunther,Sch}, where a full characterization of
affine reductions is provided (see also \cite{hajac-red}). Then, we will proceed to define quantum reductions and
examine in detail the example of the general linear quantum
group over the quantum projective space. 

The sheaf theoretic approach is then necessary to capture the notion
of principal bundle  when looking
for reduction for bundles over non affine bases, as the case of parabolic
geometries and reductions  of the structure group from the parabolic to its Levi.  

\medskip
The organization of this paper is as follows.

\medskip
In \S \ref{class-sec} we briefly recap the classical setting,
discussing a simple example, elucidating the key conceptual steps in our subsequent treatment.

In \S \ref{qpb-sec} we lay down the main definitions and establish
our notation. We start from a sheaf theoretic definition of
{\sl quantum principal bundle} as in \cite{afl}, then building the  
notion of reduction (Def. \ref{qred-def}) from a local affine model (Def. \ref{affine-red}).

In \S  \ref{ex-sec}  we construct examples of  reductions starting from the Takhtajan-Sudbery algebra of the quantum general linear group. These are given by quantum principal bundles over projective spaces described by sheaves of $\cO_q(P)$ and $\cO_q(G_0)$ comodule algebras with $P$ being  the maximal parabolic subgroup of $\rGL(n)$ and $G_0$ its Levi.

Finally, in \S  \ref{red-sec}, we provide a criterion (Thm. \ref{prop:gunther-sheaf})
giving a reduction of a given quantum principal bundle  (as a sheaf), based on the
affine (local) counterpart  \cite[Thm. 4]{gunther}.

\bigskip

{\bf Acknowledgements.} We thank Paolo Aschieri, Andreas  \v{C}ap and Thomas Weber for helpful discussions.
This research was supported by Gnsaga-Indam, by
COST Action CaLISTA CA21109, HORIZON-MSCA-2022-SE-01-01 CaLIGOLA, MSCA-DN CaLiForNIA - 101119552,
PNRR MNESYS, PNRR National Center for HPC, Big Data and Quantum Computing, INFN Sezione Bologna.  
CP is grateful to the Department of Mathematics, Informatics and Geosciences of Trieste University for the hospitality.

\section{The classical setting}\label{class-sec}
We recall the definition and some classical results about
restriction and prolongation  of principal bundles, following 
\cite[Chs. 2 and 6]{hus}.

\begin{definition}
Let $E$ and $M$ be topological spaces, $P$ a topological group
and $\pi: E \lra M$ a surjective continuous function. The data $\xi=(E, \pi, M)$
is a $P$-\textit{principal bundle} with \textit{total space} $E$ and \textit{base}
$M$, if
\begin{enumerate}
\item $P$ acts freely from the right on $E$;
\item $P$ acts transitively on the fiber $\pi^{-1}(m)$ of each point $m \in M$.
\end{enumerate}
\end{definition}

From now on, assume $K$ to be a closed subgroup of a group $P$. We will sometimes denote a bundle just by $\pi:E\to M$ or simply $E\to M$.

\begin{definition}\label{red-def}
Let $\xi=(E, \pi, M)$ be a principal $P$-bundle and $\xi_0=(E_0, \pi_0, M)$
a principal $K$-bundle over same base $M$.  The bundle 
$\xi_0$ is a \textit{restriction}, or \textit{reduction}, of $\xi$ if
there is a $K$-equivariant homeomorphism
$$
\varphi:E_0\to \phi(E_0) \subset E, \qquad \varphi(x k)= \varphi(x) k,
\quad   x \in E_0, \, k \in K,
$$
onto a closed subset $\varphi(E_0)$ of $E$. We equivalently say that 
$\xi$ is \textit{reducible} to $\xi_0$ and we call
$\xi$ a \textit{prolongation} of $\xi_0$.
\end{definition}

\medskip

On the one hand, it is known that a prolongation of a principal bundle always exists. In fact
if $\xi_0=(E_0, \pi_0, M)$ is a 
$K$-principal bundle, for $K$ a closed subgroup of $P$,
one can always construct
a $P$-principal bundle $\xi$ which is a prolongation of $\xi_0$.
Consider $(E_0 \times P) \slash K$ 
defined out of the following $K$-action,
$$
(x,p) k= (xk, k^{-1} p), \qquad
\hbox{for} \, x \in E_0, \, k \in K \, \hbox{and} \,
p \in P,
$$
and the (well-defined) projection map 
$\pi: (E_0 \times P) \slash K \to  M, [(x,p)] \to \pi_0(x)$. Then
$\xi:=((E_0 \times P) \slash K, \pi, M)$ is a $P$-principal bundle which is a prolongation of $\xi_0$ with morphism $\varphi: E_0 \to (E_0 \times P) \slash K$, $x \mapsto [(x,e)]$. 

On the other hand, the restriction is not always possible; for example
for  $K=\{e\} \subset P$, a principal $P$--bundle admits a restriction to a $K$-principal bundle
if and only if it is trivial.

Using  the same  notation  as above, the following classical result holds  (see e.g. \cite[Ch. 6, Thm. 2.3]{hus}):

\begin{proposition} \label{thm-class-red}
A principal $P$--bundle $\xi=(E, \pi, M)$
is reducible to a principal $K$--bundle $\xi_0=(E_0, \pi_0, M)$ 
if and only if the
bundle $ \xi_K := (E \setminus K ,\pi_K, M)$,   with ${\pi}_K$ being the projection induced by $\pi$ on $E \setminus K$, 
admits a global section.
\end{proposition}
We recall that  the bundle $ \xi_K$  
is isomorphic to the  bundle associated to $\xi$ via the
representation of $P$ on the space $P \setminus K$ of
right cosets of $K$ in $P$
(constructed out of left $K$-action  on $P$).
Thus a section of  $\xi_K$  
can also be given
by a $P$-equivariant map $f:E \to P \setminus K$.

Explicitly, given such a map $f$, the reduced subbundle of $\xi$ can be recovered as
$E_0 := f^{-1} ([e])$.
Vice versa, having a $K$-reduction $E_0$ induced by the $K$-equivariant bundle morphism $\varphi:E_0 \to E$, the composition of $\phi$ with the projection $E \to E \setminus K$ descends to  a 
 section $M \to  E \setminus K$.

Notice that we can consider
algebraic, analytic or smooth principal $P$--bundles,  their reductions
and prolongations, by taking
objects and morphisms in suitable categories.

\subsection{Example of a reduction} \label{example-red}
We now discuss an enlightening  example.

Consider the special linear group $\rSL_2(\C)$ and the principal bundle
 $\rSL_2(\C) \to M= \bP^{1}(\C)=\rSL_2(\C)/B$, where $B$ is the Borel subgroup of upper triangular matrices; construct then the bundle
\beq \label{red-ex}
 \pi_0:\rSL_2(\C)/N \lra \bP^{1}(\C)
\eeq
where
$N \subset B$ is the nilpotent subgroup. In the above notation we have:
$$
E=\rSL_2(\C)\, , \quad P=B \, , \quad  K=T \subset B \, , \quad  E_0= \rSL_2(\C)/N
$$
where $T$ is the torus of diagonal matrices.

In particular we have that  (\ref{red-ex}) is a $T$ principal bundle and that $( \rSL_2(\C)/N,\pi_0,\bP^{1}(\C))$ is an example of reduction. We
can write a generic element of $E$ as
$$
 \begin{pmatrix} a & b  \\ c & d \end{pmatrix}=
\begin{pmatrix} a & 0 \\ c & a^{-1} \end{pmatrix}
\begin{pmatrix} 1 & a^{-1}b \\ 0 & 1 \end{pmatrix} \,  
\hbox{if} \, a\neq 0,
\quad
\begin{pmatrix} a & b  \\ c & d \end{pmatrix}=
\begin{pmatrix} a & -c^{-1}  \\ c & 0 \end{pmatrix}
\begin{pmatrix} 1 & c^{-1}d \\ 0 & 1 \end{pmatrix} \,  
\hbox{if} \, c\neq 0.
$$
Note in fact that we can choose the open cover for $E=\rSL_2(\C)=U_1 \cup U_2$,
with $U_1=\{g \in E \, | \, a\neq 0\}$, $U_2=\{g \in E \, | \, c\neq 0\}$
and coordinates $(a,b,c)$ and $(a,c,d)$ on $U_1$ and $U_2$ respectively.
Explicitly, the change of coordinates on the
intersection of these open sets is:
\beq\label{change-coord}
U_1 \ni (a,b,c) \cong \begin{pmatrix} a & b  \\ c & d \end{pmatrix} \mapsto
\begin{pmatrix} a & ac^{-1}d  \\ c & ca^{-1}b \end{pmatrix} \cong (a,c,d)
\in U_2 \; .
\eeq
We aim to construct  an injective map $E_0 \lra E$ onto a closed set of $E$.
Let $\Ubar_i=p(U_i)$, under the projection
$p: \rSL_2(\C) \lra \rSL_2(\C)/N$, with $i=1,2$. We have
$$
\rSL_2(\C)/N \cong \C^2\setminus \{(0,0)\}=\{(a,c) \neq (0,0)\}
$$
and $\Ubar_1$ (resp. $\Ubar_2$) is the open set with $a\neq 0$
(resp. $c\neq 0$).
We have two well-defined local maps $\varphi_i:\Ubar_i \lra U_i$:
$$
\begin{array}{c}
\phi_1(a,c)=\begin{pmatrix} a & 0  \\ c & a^{-1} \end{pmatrix} \qquad
\phi_2(a,c)=\begin{pmatrix} a & -c^{-1}  \\ c & 0 \end{pmatrix}
\end{array}
$$
We leave to the reader the easy check that these maps are $K$  invariant and glue to a $K$ equivariant  homeomorphism $E_0 \lra E$. Hence we have an example of a reduction according to Def. \ref{red-def}.

\medskip
It is also fruitful to examine the same example in the light of Prop.  \ref{thm-class-red}.
We define: 
$$
\begin{array}{ccc}
  f_1:U_1 & \lra & P \setminus K \\[0.5em]
  (a,b,c) & \mapsto & \begin{pmatrix} 1 & a^{-1}b \\ 0 & 1 \end{pmatrix}K
\end{array}  
\qquad ; \qquad
\begin{array}{ccc}
f_2:U_2 & \lra & P \setminus K \\[0.5em]
(a,c,d) & \mapsto & \begin{pmatrix} 1 & c^{-1}d \\ 0 & 1 \end{pmatrix}K 
\end{array}
$$
By using the change of coordinate charts given above,
one verifies that they glue to a well-defined $P$-equivariant map
$f:E \lra P \setminus K$ as in Prop. \ref{thm-class-red}, hence we can conclude again that the $T$ principal bundle (\ref{red-ex}) is a reduction of the $B$ pricnipal bundle $\rSL_2(\C)/N \lra \bP^{1}(\C)$ .

\section{Quantum Principal Bundles and Reductions}\label{qpb-sec}

In this section, after a brief summary of the theory of
quantum principal bundles and reduction in the affine
setting, we generalize the definitions of those object to incorporate a sheaf theoretic point of view.
Our main references  are
\cite{afl,AFLW, cp, gunther,  hajac-red}, and \cite{mont}
for a treatment on Hopf algebras.

\subsection{Hopf Galois Extensions.}
We first recall some standard results on
comodule algebras and Hopf--Galois extensions
to establish our notation. We will assume all algebras to be over a field.

\medskip
Let $H$ be a Hopf algebra and $A$ a \textit{right $H$-comodule algebra},
that is an algebra with a coaction 
$\delta_H: A \to  A\otimes H$, $\delta_H(a)=\zero{a} \otimes \one{a}$, which
is an algebra morphism. The space of \textit{coinvariants}
$$
A^{\co H}:=\{b \in A ~|~ \delta_H(b)=b \otimes 1 \}
$$
is a subalgebra of $A$.   
The algebra extension $A^{\co H}\subset A$ is called \textit{$H$--Galois} ({\it Hopf-Galois}),
if the \textit{canonical map} 
\beq\label{can-map}
\chi : A\otimes_{A^{\co H}} A\rightarrow A \otimes H \, ,\; a'\otimes_{A^{\co H}} a \mapsto a'a_{(0)}\otimes a_{(1)}
\eeq
is   bijective.
The inverse of $\chi$ is a left $A$-module map as $\chi$ and thus
obtained via the \textit{translation map}:
\beq\label{transl-map}
\tau:H \lra A\otimes_{A^{\co H}} A, \quad \tau(h):=\tuno{h} \otimes_{A^{\co H}} \tdue{h}, \qquad
\chi^{-1}(a\otimes h):=a\tau(h) .
\eeq
We recall, for later use, the following
standard properties of the translation map holding for each $h,k \in H$:
\begin{align}
\label{p1} 
& \tuno{h} \otimes_{A^{\co H}} \zero{\tdue{h}} \otimes \one{\tdue{h}} =
\,\tuno{\one{h}} \otimes_{A^{\co H}} \tdue{\one{h}} \otimes 
\two{h}
\\[0.5em]
\label{p2} 
& \tuno{\two{h}}  \otimes_{A^{\co H}} \tdue{\two{h}} \otimes S(\one{h}) = \zero{\tuno{h}} \otimes_{A^{\co H}} {\tdue{h}}  \otimes \one{\tuno{h}}
\\[0.5em]
\label{tau-st}
& \tuno{(hk)} \otimes_{A^{\co H}} \tdue{(hk)}=\tuno{k}\tuno{h}\otimes_{A^{\co H}} \tdue{h}\tdue{k} \; .
\end{align}

An extension $A^{\co H}\subset A$ is called \textit{cleft} if one of the following equivalent condition
is verified (see e.g. \cite[Thm. 7.2.2, Thm. 8.2.4]{mont}): 
\begin{enumerate}
\item there exists a convolution invertible $H$-comodule map $\gamma: H \to A$ (\textit{cleaving map});
\item the algebra $A$ is isomorphic to a crossed product $A\simeq A^{\co H}\#_{}H$;
\item $A^{\co H}\subset A$ is $H$--Galois and has the normal basis property, that is
  $A$ is isomorphic to $ A^{\co H} \otimes  H$ as left $A^{\co H}$-module and right $H$-comodule
  (where $A^{\co H}  \otimes  H$ is a left $A^{\co H}$-module by left
  multiplication on the
  first factor and an $H$-comodule via $id \otimes \Delta$).
\end{enumerate}

An extension $A^{\co H}\subset A$ is called 
\textit{trivial}  if it admits a cleaving map $H\to A$ which is an $H$-comodule algebra map.
In this case, the inverse of the cleaving map is given by $\bar\gamma= \gamma \circ S$,
for $S$ the antipode of $H$.

A \textit{morphism of $H$-comodule algebras} $\phi:A\rightarrow A'$ is an algebra morphism
which intertwines
the $H$-coactions: $(\phi\otimes id_H)\circ \delta_H = \delta_{H}' \circ \phi$.
Any such $\phi$ maps coinvariants to coinvariants, $\phi(A^{\co H})\subset A'^{\co H}$. Moreover 
it intertwines the canonical maps, that is
$\chi' \circ (\phi\otimes\phi) = (\phi \otimes id_H)\circ \chi$. 

\medskip
We now give the definition of
principal $H$-comodule algebra, 
which represents the local version
of our sheaf theoretic definition of quantum principal bundle,
(see also \cite[\S 2, Def. 2.8]{AFLW}). From now on, we will assume $H$ to have bijective antipode.

\begin{definition}\label{princ-def}
Let $H$ be a Hopf algebra.
An $H$-comodule algebra $A$ is called  \textit{principal} 

if:
\begin{enumerate}
\item $A^{\co H} \subset A$ is $H$-Galois,
\item $A$ is a faithfully flat $A^{\co H}$-module.
\end{enumerate}
\end{definition}
If $A$ is a cleft extension, then $A$ is also principal. 

\medskip
\subsection{Quantum reductions.}
We first give the definition of a quantum reduction for
a quantum principal bundle in the affine setting,
inspired by \cite{gunther} (see also \cite{hajac-red}).

\begin{definition} \label{affine-red}
Let $A$ be a principal $H$-comodule algebra with $B := A^{\co{H}}$ and
$J$ be a Hopf ideal of $H$ such that $H$ is a principal left
$H_0$-comodule 
algebra for $H_0:=H/J$ with left action given by (projection of) the  coproduct.
Let $A_0$ be a principal $H_0$-comodule algebra with $B_0 := A_0^{\co{H_0}}$.

We say that $A_0$ is a \textit{reduction} of $A$ if
\begin{enumerate}
\item $B \cong B_0$ as algebras; 
\item there exists a surjective $H_0$-comodule
morphism, $\phi:A \lra A_0$, $\phi(B)= B_0$, where $A$ carries the induced $H_0$-coaction.
\end{enumerate}
We say that $A_0$ is an \textit{algebraic reduction}
if $\phi$ is an algebra morphism.
\end{definition}

\begin{remark}
Notice that when $\phi$ is a surjective algebra morphism,
that is $A_0$ is an algebraic reduction,
we obtain $A_0\cong A/I$, for some ideal $I$, retrieving the definition
as in \cite{gunther} and \cite{hajac-red}. 
Notice also that, in general, the requirement for the map $\phi: A \to A_0$ to be an algebra map is too restrictive, thus our more general definition of a reduction.  Indeed, in the degenerate case of $J$ and $I$ trivial,   a reduction is just a bundle automorphism and it is known that one has to relax the algebra map condition   in order to account for the bundle morphisms we also observe in the
ordinary setting, see \cite{brz, pgc}.
\end{remark}

\medskip
Definition \ref{affine-red} provides a local model for reduction in the
non affine case, as we shall presently see.
 Let us first recall the notion of quantum principal
bundle as in \cite{afl}.

\begin{definition}\label{qpb-def}
A \textit{quantum ringed space}  
$(M , \cO_M)$ is
a pair consisting of a classical topological space $M$ and 
a sheaf over $M$ of noncommutative algebras.
We say that a sheaf of $H$-comodule algebras $\cF$ 
is a \textit{ quantum principal bundle} over $(M,\cO_M)$
if there exists an open covering 
$\{U_i\}$ of $M$ such that:
\begin{enumerate}
\item $\cF(U_i)^{\co H}=\cO_M(U_i)$,
\item $\cF(U_i )$ is a principal $H$-comodule algebra.
\end{enumerate}
We 
say that $\cF$ is a {\it 
locally cleft quantum principal bundle}
if  $\cF$ is \textit{locally cleft}, 
that is, 
$\cF (U_i )$ is a cleft extension of $\cF(U_i)^{\co H}$ for each open $U_i$ of the covering.

We say the covering $\{U_i\}$ is a {\it local trivialization} for $\cF$
if $\cF (U_i )$ is a trivial extension of $\cF(U_i)^{\co H}$, that
is we have the isomorphism of algebras $\cF(U_i) \cong \cF(U_i)^{\co H} \otimes H$.
\end{definition}

Our definition is in agreement with \cite{AFLW} and \cite{cp},
however more general with respect to \cite{bm} and \cite{buachalla}.
\medskip

Notice that both bundles described in Example \ref{example-red} admit
a local trivialization with respect to the covering  $\{\Ubar_1 , \Ubar_2\}$.

\begin{definition}\label{qred-def}
Let $\cF$ and $\cF_0$ be  quantum principal bundles over the quantum ringed space $(M , \cO_M)$
for Hopf algebras $H$ and $H_0=H/J$, respectively.
We say that $\cF_0$ is a \textit{reduction} (resp. \textit{algebraic reduction})
of $\cF$ if there exists an open covering $\{U_i\}$ of $M$ such that: 
\begin{enumerate}
\item $\cF$ and $\cF_0$ satisfy (1) and (2) of Def. \ref{qpb-def}
with respect to such cover,
\item there exists an $H_0$-comodule (resp. $H_0$-comodule algebra) morphism $\varphi:\cF \lra \cF_0$ such that  
$\varphi(\cF(U_i)^{\co H})=\cF(U_i)^{\co H_0}$
for the induced coaction of $H_0=H \slash J$ on the $\cF(U_i)$.
\end{enumerate}
\end{definition}

Notice that $\cF_0(U_i)$ is a quantum
reduction of $\cF(U_i)$ according to Def. \ref{affine-red}. 

\begin{remark}
Assume $A$ is trivial cleft extension, i.e. $A \cong B\otimes H$ as $H$-comodule algebras,
$B=A^{co(H)}$. Let $H_0=H/J$ as above and assume we can write $H=H_0 \# B$. 
Then $A_0:=B \otimes H_0$ is a reduction. 
\end{remark}

\section{Examples of Reductions}\label{ex-sec}

In this section we provide examples of quantum reductions
starting from the multiparameter family $\cO_\mathbf{q}(\rGL(n))$
of deformations of the algebra
of coordinate functions on  the general linear group $\rGL(n)$
introduced in \cite{Tak, Sud1},
see also \cite{Schirr}.

\subsection{Multiparametric deformations of $\rGL(n)$.}
The coordinate algebra $\cO_\mathbf{q}(M(n))$, deformation of the algebra of coordinate functions
of $n$ by $n$ complex matrices, is generated by elements $\T_{ij}$, $i,j=1, \dots, n$, with commutation relations 
\begin{align} \label{multiGL}
\T_{i k} \T_{i l} = p_{k l} \, \T_{i l}  \T_{i k}  \; ;  \qquad
& \T_{i k} \T_{j k}  = q_{i j}\,\T_{j k}  \T_{i k}   \nn \\ 
 \T_{i l} \T_{j k}  = \frac{q_{i j}}{p_{k l}} \, \T_{j k} \T_{i l}     \; ;  \qquad
& \T_{i k} \T_{j l} = \frac{p_{k l}}{p_{i j}} \, \T_{j l}\T_{i k} + (p_{k l} - \frac{1}{ q_{k l}})\, \T_{i l} \T_{j k}  ~, \qquad i<j , k<l ~.
\end{align}
The relations  depend on the $n(n-1)/2$ non-zero parameters $q_{i j}, i< j$ and an additional parameter $u$ with 
 $p_{i j}:= \frac{u}{q_{i j}}. $
The algebra $\cO_\mathbf{q}(M(n))$ is a bialgebra with the standard matrix coalgebra structure:
$$\Delta (\T_{i j})= \sum_k \T_{i k} \otimes \T_{k j} \quad , \quad \varepsilon(\T_{i j}) =\delta_{i j}\, .$$
The element
\beq\label{det-multi}
D_\mathbf{q}:= 
\sum_{\sigma \in P_n} \epsilon^{\sigma_1 \cdots \sigma_n}\T_{1 \sigma_1}\cdots \T_{n \sigma_n}= 
  \sum_{\sigma \in P_n} \epsilon_{\sigma_1 \cdots \sigma_n}\T_{\sigma_1 1}\cdots \T_{\sigma_n n}
\eeq
is called the \textit{quantum determinant}, where
$$
\epsilon^{\sigma_1 \cdots \sigma_n}:=  \prod_{\alpha < \beta \, ; \,   \sigma_\beta < \sigma_\alpha }  (-p_{\sigma_\beta \sigma_\alpha})
\; , \quad
\epsilon_{\sigma_1 \cdots \sigma_n}:=  \prod_{\alpha < \beta \, ; \,   \sigma_\beta < \sigma_\alpha }  (-q_{\sigma_\beta \sigma_\alpha})
$$
and $\epsilon^{\sigma_1 \cdots \sigma_n}=1$,  $\epsilon_{\sigma_1 \cdots \sigma_n}=1$, when no $p$, respectively $q$, appear in the product.  
In general $D_\mathbf{q}$ is not central, nevertheless it obeys simple commutation
relations with the algebra generators:
\beq\label{comm-multi-det}
\T_{i k} D_\mathbf{q} = \frac{\prod\limits_{\alpha=1}^{k-1} q_{\alpha k}}{\prod\limits_{\beta=1}^{i-1} q_{\beta i}} \frac{\prod\limits_{\gamma=k+1}^{n} p_{k \gamma}}{\prod\limits_{\delta=i+1}^{n}p_{ i \delta}} D_\mathbf{q} \T_{i k} .
\eeq
The algebra $\cO_\mathbf{q}(\rGL(n))$ of the quantum general linear group   is then introduced as the algebra extension 
of $\cO_\mathbf{q}(M(n))$ by the element $D_\mathbf{q}^{-1}$, the inverse of the determinant,  with 
 commutation relations deduced  from those of  $D_\mathbf{q}$ in \eqref{comm-multi-det} imposing compatibilty with $D_\mathbf{q}^{-1} D_\mathbf{q}= D_\mathbf{q} D_\mathbf{q}^{-1}=1$.
 The quantum determinant and its inverse are
 group-like,  
 $$
 \Delta (D_\mathbf{q}^{\pm 1} ) =D_\mathbf{q}^{\pm 1}  \otimes D_\mathbf{q}^{\pm 1}  \quad  ,\quad  \varepsilon (D_\mathbf{q}^{\pm 1} )=1  .
 $$
The bialgebra $\cO_\mathbf{q}(\rGL(n))$ becomes a Hopf algebra with antipode $S$ defined in terms of quantum cofactors, see e.g. \cite{Sud1}.

\begin{example} 
In the lowest dimensional case $n=2$, the algebra $\cO_\mathbf{q}(\rGL(2))$  is generated by elements $a:=\T_{11}$, $b:=\T_{12}$, $c:=\T_{21}$, $d:=\T_{22}$ and $D_\mathbf{q}^{-1}$ that satisfy 
\begin{align} \label{multi2}
& a b = p \, b a \; ,   
& a c  = q \, c a  \; , \quad
&& c d = p \, d c \nn \\
& b d = q \, d b  \; , 
& b c  = \frac{q}{p} \, c b    \; , \quad 
&& a d =   d a + (p - \frac{1}{ q})\, b c \, ,
\end{align}
for $p=u q^{-1}$, with $a$ and $d$ commuting 
 with the quantum determinant $ D_\mathbf{q} = ad - p \, bc$ and its inverse  $D_\mathbf{q}^{-1}$ and
$$
b D_\mathbf{q}^{\pm 1}  = \left(\frac{q}{p}\right)^{\pm 1}  D_\mathbf{q}^{\pm 1}  b \; , \qquad
c D_\mathbf{q}^{\pm 1}  = \left(\frac{q}{p} \right)^{\mp 1}  D_\mathbf{q}^{\pm 1} c \; . 
$$
The antipode is the anti-algebra map given in matrix form as
$$
S \begin{pmatrix}
a & b \\ c & d 
\end{pmatrix}
=
\begin{pmatrix}
d & - \frac{1}{p} b \\ - p c & a 
\end{pmatrix} D_\mathbf{q}^{-1}
= D_\mathbf{q}^{-1}
\begin{pmatrix}
d & - \frac{1}{q} b \\ - q c & a 
\end{pmatrix}
$$
and $S(D_\mathbf{q}^{\pm 1})= D_\mathbf{q}^{\mp 1}$.
\end{example}

\subsection{The Takhtajan-Sudbery algebra.}
We now address our attention to the  one-parameter  deformation corresponding to the choice  $u=1$ and $q_{ij}=q  \in \mathbb{C} \slash \{0\} $ for all $i<j$ (and hence $p_{ij}=q^{-1}$  for all $i<j$).  We refer to this algebra as the Takhtajan-Sudbery algebra.  
We observe that, for this choice of parameters, the determinant is  not central (and thus the quotient to quantum special linear groups is not defined),  while $\T_{i k} \T_{j l}=\T_{j l}\T_{i k} $ hold,
for all $i<j , k<l$. In the following 
we simply write $D^{\pm 1}$  for the quantum determinant and its inverse.
 We denote this Hopf algebra by $\widetilde{\cO}_q(\rGL(n))$, to distinguish it from the `standard'  quantum group $\cO_q(\rGL(n))$ of FRT bialgebras,
which corresponds to the choice  $u=q^2$ and $q_{ij}=q$, for all $i<j$.
For this choice of the parameters, indeed, the quantum determinant is central and one can construct the  quotient Hopf algebra  $\cO_q(\rSL(n)):=\cO_q(\rGL(n))/(D-1)$ whose prototype is the quantum group $\cO_q(\rSL(2))$ \cite{manin}.  
\medskip

The algebra
$\widetilde{\cO}_q(\rGL(n))$ is generated by elements $\T_{ij}$, $i,j=1, \dots, n$, with commutation relations given in \eqref{multiGL} (for $q_{ij}=q$ and $p_{ij}=q^{-1}$  for all $i<j$):
\beq\label{multiGL-one}
\T_{i k} \T_{i l} = q^{-1} \, \T_{i l}  \T_{i k}  \, ;  \quad
 \T_{i k} \T_{j k}  = q \,\T_{j k}  \T_{i k}   \, ;  \quad
 \T_{i l} \T_{j k}  = q^2\, \T_{j k} \T_{i l}     \, ;  \quad
 \T_{i k} \T_{j l} =   \T_{j l}\T_{i k} \, , \; i<j , k<l
\eeq
together with $D^{-1}$, the inverse of the quantum determinant $D$,
with commutation relations
\beq\label{comm-det-one}
\T_{i k} D =  q^{2(k-i)} D  \T_{i k} \; , \quad 
\T_{i k} D^{- 1} =  q^{-2(k-i)} D^{-1} \T_{i k} ~ .
\eeq
Observe that we can rewrite the above commutation relations  in a condensed form as  
\beq\label{commGL-mu}
\T_{i k} \T_{j l} = \mu_{ij} \mu_{lk} \, \T_{j l}  \T_{i k}  \, , \quad \qquad
\T_{i k} D^{\pm 1} = \prod\limits_{\alpha=1}^{n} (\mu_{\alpha k} \mu_{ i\alpha})^{\pm 1} D^{\pm 1} \T_{i k} \qquad
\forall ~  i,j,k,l 
\eeq
with
\beq\label{mu-coeff}
\mu_{ij}:= 
\left\{ 
\begin{array}{ll}
q & \mbox { if } i<j 
\\
1 & \mbox { if } i=j 
\\
q^{-1} & \mbox { if } i>j ~ .
\end{array}
\right. 
\eeq
By definition, $\mu_{i j} \mu_{j i}=1$, for all $i,j=1, \dots , n$ and moreover for each index $\ell$ fixed
\beq\label{prop2}
\mu_{i \ell} \mu_{\ell j}=1 \; \mbox{ if both } i,j <\ell \mbox{ or   } i,j >\ell .
\eeq
\medskip

\subsection{The Quantum ringed space $(\bP^{n-1}, \cO_{\bP^n})$.}
In \cite{cfg}, through the concept of {\sl quantum section}, a construction of a quantum
deformation of the homogeneous coordinate ring $\cO_q(G/P)$ of the projective variety  
$G/P$ appears, for $G$ a complex algebraic group and $P$ a parabolic subgroup. The quantum section
is the quantum equivalent of the character of $P$ encoding the information on the projective embedding of $G/P$,
via the line bundle on $G/P$ defined through the given character.
Such embedding is giving the homogeneous commutative ring $\cO(G/P)$, recovered
in the space of the global sections of the line bundle.
For more details 
see \cite{cfg, afl}.

\medskip
We recall the key result of \cite{cfg} that we need here.
Let $\cO_q(G)$ be a Hopf algebra, with comultiplication $\Delta$ and counit $\epsilon$,
which is a quantum deformation  of  the coordinate ring $\cO(G)$ of the algebraic group $G$.

\begin{theorem}  \label{Oqgh-graded}
Let  $ d \in \cO_q(G)$  be a quantum section i.e.
$(id \otimes \pi) \Delta(d) = d \otimes \pi(d)$. Then
\begin{enumerate}
\item $ \cO_q(G/P)={\oplus}_{n \in \N} \cO_q(G/P)_n $  is a  graded subalgebra  of
$ \, \cO_q(G) \,$, where 
$$
\cO_q(G/P)_n = \{f \in \cO_q(G)\, | \, 
(\text{\it id} \otimes \pi)\Delta(f) = f \otimes \pi\big(d^n\big)\}
$$
for $\pi: \cO_q(G) \lra \cO_q(P):=\cO_q(G)/I_P$, $\cO_q(P)$ a quantum group deformation of  the coordinate ring of
the parabolic subgroup $P\subset G$.

\item 
$ \, \cO_q(G/P) $  is a  {\sl graded}
$ \, \cO_q(G) $--comodule  algebra, via
the restriction of the comultiplication $\Delta$
in $\cO_q(G)$, 
$$
\Delta|_{\cO_q(G/P)}: \cO_q(G/P) \lra \cO_q(G) \otimes 
\cO_q(G/P)~.
$$
\end{enumerate}                               
\end{theorem}

In our special setting, we take $\cO_q(G)=\widetilde{\cO}_q(\rGL(n))$ and one
can check immediately that $d=a_{11}$ is a quantum section,
for
\beq\label{par-sub}
\pi:\cO_q(G)\lra \cO_q(P):=:\cO_q(G)/(a_{s \, 1},\, s=2, \dots, n) \; .
\eeq
With this
choice, by Thm \ref{Oqgh-graded}, we obtain
$\cO_q(G/P)$ as the graded subalgebra of $\cO_q(G)$ generated by the $a_{\ell \, 1},\, \ell=1, \dots, n$.
In terms of generators and relations, we can express $\cO^q_{\bP^{n-1}(\C)}:=\cO_q(G/P)$ as the
graded non commutative algebra generated by
generic variables $x_i$'s, $i=0, \dots, n-1$, subject to the (homogeneous) relations:
$$
x_ix_j=qx_jx_i, \qquad 0 \leq i<j \leq n-1 \; .
$$

We now construct a sheaf of noncommutative algebras $\cO^q_{\bP^{n-1}(\C)}$. We consider
the classical affine open cover of $\bP^{n-1}(\C)$, defined as:
\beq\label{aperti-U}
U_\ell:=\{[\hat{x}_0, \dots, \hat{x}_{n-1}]\,| \, \hat{x}_\ell \neq 0 \}\subset \bP^{n-1}(\C)
\eeq
where $[\hat{x}_0, \dots, \hat{x}_{n-1}]$ represents a point in $\bP^{n-1}(\C)$.
We then consider the topology generated by the $U_\ell$'s; the intersections
of an arbitrary number of $U_\ell$'s  
form a basis.  In analogy with the classical setting, we define $\cO^q_{\bP^{n-1}(\C)}(U_\ell)$ as
the non commutative {\sl projective localization} of $\cO^q_{\bP^{n-1}(\C)}$ at the multiplicatively
closed set $S_\ell=\{x_\ell^n, n \in \mathbb{N}\}$:
$$
\cO^q_{\bP^{n-1}(\C)}(U_\ell):=[\cO^q_{\bP^{n-1}(\C)}S_\ell^{-1}]_{\mathrm{proj}} \, .
$$
This construction amounts to take the localization at
$S_\ell$ and then the subalgebra of degree zero elements. This construction is standard, see \cite[Ch. 4]{fl}
 for more details. We leave to the reader the checks; notice that all elements of $S_\ell$ are Ore.
We have the following result, that we shall use later.

\begin{proposition}
Let the notation be as above. Then,
\begin{enumerate}
\item we can express $\cO^q_{\bP^{n-1}(\C)}(U_\ell)$ in terms of generators
and relations as:
\beq\label{sheaf-pnc}
\cO^q_{\bP^{n-1}(\C)}(U_\ell) = \C[x^{\scriptscriptstyle{(\ell)}}_i]_{i =1, \dots , n; ~i \neq \ell}/(x^{\scriptscriptstyle{(\ell)}}_i x^{\scriptscriptstyle{(\ell)}}_j - \mu_{\ell i} \mu_{i j} \mu_{j \ell}  ~ x^{\scriptscriptstyle{(\ell)}}_j x^{\scriptscriptstyle{(\ell)}}_i ) 
\eeq
\item for each $\ell=1, \dots, n$, the  algebra $\cO^q_{\bP^{n-1}(\C)}(U_\ell)$
is isomorphic to the algebra \\ $\cO^q_{\bP^{n-1}(\C)}(U_1)$ via the map
$$
\cO^q_{\bP^{n-1}(\C)}(U_1) \to \cO^q_{\bP^{n-1}(\C)}(U_\ell) \, ,\quad  x^{\scriptscriptstyle{(1)}}_j \mapsto \left\{ 
\begin{array}{ll}
 x^{\scriptscriptstyle{(\ell)}}_{j+ \ell -1} & \mbox{if }~ j=2, \dots , n-\ell+1 ~;
 \\
 \\
x^{\scriptscriptstyle{(\ell)}}_{j+\ell -1-n} & \mbox{if } ~ j=n-\ell+2, \dots , n ~.
\end{array}
\right. 
$$
\end{enumerate}
\end{proposition}
\begin{proof}
(1). We notice that $x^{\scriptscriptstyle{(\ell)}}_j := x_j x_\ell^{-1}$ generate $\cO^q_{\bP^{n-1}(\C)}(U_\ell)$
and satisfy the given commutation relations.
\medskip
(2). The proof is completely analogous to that of Proposition \ref{prop:b1=bl} and thus we omit it.
\end{proof}

Let us denote by $AS^{-1}$ the localization of the algebra $A$ at the multiplicatively
closed subset $S$ (of Ore elements).

\begin{proposition}
Let the notation be as above.
The assignment
$$
U_{i_1} \cap \dots \cap U_{i_r} \mapsto \cO^q_{\bP^{n-1}(\C)}(U_\ell)S_{i_1}^{-1} \dots S_{i_r}^{-1} 
$$
defines a sheaf of non commutative algebras $\cO^q_{\bP^{n-1}(\C)}$ on $\bP^{n-1}$, with respect
to the topology defined by the open cover $\{U_\ell\}$.
Hence $(\bP^{n-1}(\C),\cO^q_{\bP^{n-1}(\C)})$ is a quantum ringed space.
\end{proposition}
\begin{proof} The $U_{i_1} \cap \dots \cap U_{i_r}$ form a basis for the topology,
restriction maps are naturally defined via localizations. We leave to the reader
all remaining checks.
See \cite[Observation 4.4]{AFLW}.
\end{proof}

\subsection{The Quantum Principal bundle $\cF$.}
We now construct a quantum principal bundle $\cF$,
which is a deformation of the principal bundle $\rGL(n) \lra \rGL(n)/P
\cong \bP^{n-1}(\C)$. 
Classically, it is the same construction as in \cite{afl}.
However, since here the deformation is substantially different,
we need to take care of the technicalities involved in
verifying that we indeed have a quantum principal bundle.
We shall make an essential use of 
\cite[Thm 4.8]{afl}. Later on, in \S \ref{sec:qrF}, we will construct
  a quantum reduction of $\cF$.

\medskip
Let $A=\widetilde{\cO}_q(\rGL(n))$.
For each $\ell=1, \dots, n$, consider the algebra extension $A_\ell:= A[\T_{\ell 1}^{-1}]$ of $A$ by the inverse of $\T_{\ell 1}$,  
$\T_{\ell 1}^{-1} \T_{\ell 1} = 1 = \T_{\ell 1} \T_{\ell 1}^{-1}$.
In view of our application of  \cite[Thm 4.8]{afl},
  the $\T_{\ell 1}$ here are the $d_\ell$'s there, while the quantum section is 
  $a_{11} \in A$.

For consistency, the additional generator satisfies
commutation relations
\beq\label{comm-rel-a-ainv}
\T_{j k} \T_{\ell 1}^{-1} =\mu_{k 1 }  \mu_{ \ell j}   \, \T_{\ell 1}^{-1}   \T_{j k}  \, , \qquad  \;  j,k =1, \dots, n
.
\eeq

\begin{proposition}\label{prop:iso}
Let $\ell$, $m$ two fixed indices in $\{1, \dots , n\}$ and 
  $A_\ell $ and $A_m$ the corresponding algebra extensions. Then, the map  (change of coordinates)
\begin{align}\label{psi-lm}
\Psi_{\ell m}:
A_{\ell}[\T_{m 1}^{-1}]  & \stackrel{\simeq}{\longrightarrow} A_{m}[\T_{\ell 1}^{-1}]
\\ \nn
\T_{\ell k} & \mapsto \; \T_{\ell 1} \T_{m1}^{-1} \T_{mk} 
\\ \nn
\T_{\ell 1}^{-1} & \mapsto  \; \T_{\ell 1}^{-1}
\\ \nn
\T_{mk} & \mapsto \; \T_{m 1} \T_{\ell  1}^{-1} \T_{\ell  k} 
\\ \nn
\T_{m 1}^{-1} & \mapsto \; \T_{m 1}^{-1}
\\ \nn
\T_{j k} & \mapsto \; \T_{j k}  \; , \qquad \forall ~ j,k=1, \dots, n , ~ j \neq \ell , m
\\ \nn
D^{-1} & \mapsto \; D^{-1}
\end{align}
is an algebra isomorphism. 
\end{proposition}
\begin{proof}
The proof is by direct computation, using the commutation relations given in  concise forms in   \eqref{commGL-mu} and \eqref{comm-rel-a-ainv} and the property $\mu_{i j} \mu_{j i}=1$, for all $i,j=1, \dots , n$.
\end{proof}
In matrix form,
in  the algebra  $A_\ell$
we write:
$$
(\T_{ij}) =  \begin{pmatrix}
\T_{11} & \gu{\ell}_{12} & \dots & \gu{\ell}_{1n}
\\
\T_{21} & \gu{\ell}_{22} & \dots & \gu{\ell}_{2n}
\\
\vdots
\\
\T_{\ell 1} & 0 & \dots & 0
\\
\vdots
\\
\T_{n1} & \gu{\ell}_{n2} & \dots & \gu{\ell}_{nn} 
\end{pmatrix}
\begin{pmatrix}
1 & \gn{\ell}_{12} & \gn{\ell}_{13} &\dots & \gn{\ell}_{1n}
\\
0 & 1& 0& \dots & 0
\\
&& \ddots &&
\\
0& 0& \dots &1 & 0
\\
0& 0& \dots &0 & 1
\end{pmatrix}
$$
that is
\beq
\gn{\ell}_{1j}:= \T_{\ell 1}^{-1} \T_{\ell j} 
\; , \qquad
\gu{\ell}_{ij}:= \T_{i j}- \T_{i 1} \T_{\ell 1}^{-1} \T_{\ell j}  ~, \;~~j>1~ i \neq \ell ~.
\eeq

\begin{lemma}
For $\ell \in \{1, \dots , n\}$ fixed,   
\beq\label{Aell}
A_\ell = \C_q[\T_{k1}, \T_{\ell 1}^{ - 1}, \gn{\ell}_{1s}, \gu{\ell}_{js}, D^{\pm 1}]_{j,k=1, \dots , n , ~ j \neq \ell, ~ s= 2, \dots , n} \slash I_\ell
\eeq
 for $I_\ell$ the ideal giving commutation relations
 \begin{align}\label{comm-rel-nu}
 &\T_{k 1} \T_{j1} =   \mu_{kj} \T_{j1}  \T_{k 1} 
  &&\T_{k 1} \T_{ \ell 1}^{ -1} =   \mu_{\ell k} \T_{ \ell 1}^{ -1}  \T_{k 1} 
 \nn \\
&\T_{k 1} \gn{\ell}_{1s} =   \mu_{s1} \gn{\ell}_{1s}  \T_{k 1} 
&& \T_{k 1} \gu{\ell}_{is} =  \mu_{s1} \mu_{ki}  \gu{\ell}_{is} \T_{k 1} 
\nn \\
&\T_{\ell 1}^{ - 1} \gn{\ell}_{1s} =  \mu_{1s} ~ \gn{\ell}_{1s} \T_{\ell 1}^{-1} ;
&& \T_{\ell 1}^{ - 1} \gu{\ell}_{is} =    \mu_{1s}  \mu_{i \ell } \gu{\ell}_{i s}\T_{\ell 1}^{-1} ;
\nn \\ 
& \gn{\ell}_{1s}  \gn{\ell}_{1t}= \mu_{ts}  \gn{\ell}_{1t}\gn{\ell}_{1s}
&& \gn{\ell}_{1s} \gu{\ell}_{kt}   =\mu_{1t}  \mu_{ts} \gu{\ell}_{kt} \gn{\ell}_{1s} 
\nn \\
&\gu{\ell}_{js} \gu{\ell}_{kt}= \mu_{jk} \mu_{ts} \gu{\ell}_{kt}\gu{\ell}_{js}
\end{align}
and
\begin{align}\label{comm-rel-nu2}
& \T_{k 1} D^{\pm 1} = \prod\limits_{\alpha=1}^{n}\big( \mu_{\alpha 1} \mu_{ k\alpha}\big)^{\pm 1} D^{\pm 1}  \T_{k 1} 
&& \T_{\ell 1}^{ - 1} D^{\pm 1} = \prod\limits_{\alpha=1}^{n}\big( \mu_{\alpha 1} \mu_{ \ell\alpha}\big)^{\mp 1}D^{\pm 1} \T_{\ell 1}^{ - 1}  \nn
\\
& \gn{\ell}_{1s} D^{\pm 1} =  \prod\limits_{\alpha=1}^{n}\big( \mu_{ 1\alpha} \mu_{\alpha s}\big)^{\pm 1}  D^{\pm 1}  \gn{\ell}_{1s} 
&& \gu{\ell}_{kt} D^{\pm 1} =  \prod\limits_{\alpha=1}^{n}\big( \mu_{\alpha t} \mu_{ k\alpha}\big)^{\pm 1} D^{\pm 1} \gu{\ell}_{kt}
\end{align}
for all $j,k=1 , \dots ,n$ and $s,t=2, \dots , n$.
\end{lemma}
\begin{proof}
The commutation relations  above are computed directly  from \eqref{commGL-mu}, \eqref{comm-det-one} and   \eqref{comm-rel-a-ainv}.
\end{proof}
In the identification \eqref{Aell}, the algebra isomorphism  
$ \Psi_{\ell m}:A_{\ell}[\T_{m 1}^{-1}] \to A_{m}[\T_{\ell 1}^{-1}]$ in Proposition \ref{prop:iso}  is
 given on the algebra generators by 
 \beq\label{psi-lm2}
 \T_{j 1} \mapsto \T_{j 1} \, ,\;
 \T_{\ell 1}^{ - 1} \mapsto \T_{\ell 1}^{ - 1} \, ,\;
  \T_{m 1}^{ - 1} \mapsto \T_{m 1}^{ - 1} \, ,\;
\gn{\ell}_{1j} \mapsto  \gn{m}_{1j} \, ,\;
\gu{\ell}_{mj} \mapsto  \T_{m1} \T_{\ell 1}^{-1} \gu{m}_{\ell j} \, ,\;
\gu{\ell}_{ij} \mapsto \gu{m}_{ij}
\eeq
for all $i,j=1 , \dots , n$, $i \neq \ell,m$.
\\

The algebra $A=\widetilde{\cO}_q(\rGL(n))$ is an $H$-comodule algebra for the Hopf algebra 
 $H:=\widetilde{\cO}_q(P):=\widetilde{\cO}_q(\rGL(n))/(\T_{s1}, s \neq 1)$, with  right coaction given by the composition of the coproduct with the quotient map $pr: A \to H$:
 \beq\label{coac-n}
 \delta :=(id \otimes pr)\circ \Delta : A \to A \otimes H \, , \quad 
 \T_{ij} \mapsto \sum_{k=1}^n \T_{ik} \otimes p_{kj}\;   ,\quad
 D^{-1} \mapsto D^{-1}   \otimes \tilde{D}^{-1}  \; .
 \eeq 
 Here 
 $p_{ij}, \tilde{D}^{-1}$ denote the images of the generators $\T_{ij}, {D}^{-1}$ of $A$ under the quotient map.  The generators of the Hopf algebra  $\widetilde{\cO}_q(P)$ satisfy
\beq\label{comm-p}
p_{i k} p_{j l} = \mu_{ij} \mu_{lk} \, p_{j l}  p_{i k}  \, , \quad 
p_{i k} \tilde{D}^{\pm 1} = \prod\limits_{\alpha=1}^{n} (\mu_{\alpha k} \mu_{ i\alpha})^{\pm 1} \tilde{D}^{\pm 1} p_{i k}\, , \quad  
\forall ~  i,j,k,l  \; .
\eeq 
Moreover (cf. \eqref{det-multi}),
\beq 
\tilde{D}= p_{11} M_{11}= M_{11} p_{11}, \quad M_{11}:=\sum_{\sigma \in P_{n-1}} \epsilon_{\sigma_2 \cdots \sigma_n}p_{\sigma_2 2}\cdots p_{\sigma_n n}
\eeq
 with 
$\epsilon_{\sigma_2 \cdots \sigma_n}:=  \prod_{\alpha < \beta \, ; \,   \sigma_\beta < \sigma_\alpha }  (-q_{\sigma_\beta \sigma_\alpha})
$
and  $\epsilon_{\sigma_2 \cdots \sigma_n}=1$, when no   $q$ appear in the product.  Thus $p_{11}$ is an invertible group-like element, with $p_{11}^{-1}=M_{11}\tilde{D}^{-1}$.\\

Being $\delta( \T_{\ell 1})= \T_{\ell 1} \otimes p_{11}$,  the coaction $\delta$ in \eqref{coac-n} extends  as an algebra map in a unique way  to coactions $\delta_\ell: A_\ell \to A_\ell \otimes H$  on the algebra extensions $A_\ell= A[\T_{\ell 1}^{-1}]$,   $\ell=1, \dots , n$:
 \beq\label{coac-n2}
 \delta_{\ell} (\T_{\ell 1}^{-1}) =\T_{\ell 1}^{-1} \otimes p_{11}^{-1} ~ .
 \eeq
In  the identification $A_\ell= \C_q[\T_{k1}, \T_{\ell 1}^{ - 1}, \gn{\ell}_{1s}, \gu{\ell}_{js}, D^{\pm 1}]  \slash I_\ell$ in \eqref{Aell}, 
 \beq\label{coact-nu}
 \delta_{\ell} (\gn{\ell}_{1s})= \sum_{r=2}^n \gn{\ell}_{1r} \otimes p_{11}^{-1} p_{rs} 
 \; , \quad
  \delta_{\ell} (\gu{\ell}_{js})= \sum_{r=2}^n\gu{\ell}_{j r} \otimes   p_{r s} 
   \; , \quad j \neq \ell, ~ s>1.
 \eeq
For each $\ell$, the corresponding subalgebra of coinvariants $B_\ell:= A_\ell^{co H}$ is generated by the elements  
 \beq\label{coinv-gln}
 \gb{\ell}_j:= \T_{\ell 1}^{-1} \T_{j 1} \; , \quad j=1, \dots, n \, , \; (j \neq \ell). 
 \eeq
They satisfy commutation relations
\beq\label{comm-b}
  \gb{\ell}_i   \gb{\ell}_j = \mu_{\ell i} \mu_{i j} \mu_{j \ell}  ~  \gb{\ell}_j  \gb{\ell}_i
\eeq
In particular, for $\ell=1$, the generators of the algebra $B_1$ satisfy commutation relations
\beq\label{comm-b1}
  \gb{1}_i  b^{\scriptscriptstyle{(1)}}_j
 =  \mu_{i j}   \gb{1}_j \gb{1}_i ,
\eeq
being,  from \eqref{prop2}, $\mu_{1 i}\mu_{j 1}=1$ for all $i,j=2, \dots, n$.
 \medskip

\begin{proposition} \label{prop:b1=bl}
For each $\ell=1, \dots, n$ the  algebra $B_\ell$ is isomorphic to the algebra $B_1$ via the map
\beq
\beta_\ell: B_1 \to B_\ell \, ,\quad   \gb{1}_j \mapsto \left\{ 
\begin{array}{ll}
 \gb{\ell}_{j+ \ell -1} & \mbox{if }~ j=2, \dots , n-\ell+1 ~;
 \\
 \\
 \gb{\ell}_{j+\ell -1-n} & \mbox{if } ~ j=n-\ell+2, \dots , n ~.
\end{array}
\right. 
\eeq
\end{proposition}
\begin{proof}
If $j,k=2, \dots , n-\ell+1$, then 
\begin{align*}
\beta_\ell(\gb{1}_j)\beta_\ell(\gb{1}_k)
&=
 \gb{\ell}_{j+ \ell -1}    \gb{\ell}_{k+ \ell -1} 
 =
 \mu_{\ell, j+ \ell -1} \mu_{j+ \ell -1, k+ \ell -1} \mu_{k+ \ell -1 , \ell}  ~  \gb{\ell}_{k+ \ell -1}  \gb{\ell}_{j+ \ell -1}
  \\
&= 
 \mu_{jk}  \gb{\ell}_{k+ \ell -1}  \gb{\ell}_{j+ \ell -1}
=  \mu_{jk}  \beta_\ell(\gb{1}_k)\beta_\ell(\gb{1}_j)
 \end{align*}
where we used that $j+\ell-1, k+\ell -1> \ell$ and thus $\mu_{\ell, j+ \ell -1} \mu_{k+ \ell -1 , \ell} =1 $ because of \eqref{prop2} and $ \mu_{j+ \ell -1 , k+ \ell -1} =  \mu_{j k }$.  Analogous computation is done for $j,k=n-\ell+2, \dots , n$. 
Finally, if $j=2, \dots , n-\ell+1$ and $k=n-\ell+2, \dots , n$,  
\begin{align*}
\beta_\ell( \gb{1}_j)\beta_\ell( \gb{1}_k)
&=
  \gb{\ell}_{j+ \ell -1}    \gb{\ell}_{k+ \ell -1-n} 
 =
 \mu_{\ell , j+ \ell -1} \mu_{j+ \ell -1 , k+ \ell -1-n} \mu_{k+ \ell -1 -n, \ell}  ~  \gb{\ell}_{k+ \ell -1-n}  \gb{\ell}_{j+ \ell -1}
  \\
&= 
\mu_{jk} ~  \gb{\ell}_{k+ \ell -1-n}  \gb{\ell}_{j+ \ell -1}
=  \mu_{jk}  \beta_\ell( \gb{1}_k)\beta_\ell( \gb{1}_j)
 \end{align*}
being here $j+ \ell -1> \ell>k+ \ell -1-n $ and thus $\mu_{jk}= \mu_{\ell ,j+ \ell -1} = \mu_{k+ \ell -1 -n, \ell}= \mu_{j+ \ell -1 , k+ \ell -1-n}^{-1}$ because of \eqref{mu-coeff}.
\end{proof}

 \begin{proposition}\label{prop:HG}
 For each $\ell=1, \dots, n$,
the algebra extension $B_\ell \subset A_\ell$  is a faithfully flat $H$-Galois extension.
 \end{proposition}
 \begin{proof} 
 Since $H=\widetilde{\cO}_q(P)$ is cosemisimple and has
bijective antipode,  the surjectivity of the canonical map implies bijectivity and faithfully
flatness of the extension, \cite[Th. I]{sch90}. 
Moreover, even if the canonical map and thus its inverse are not in general algebra maps, in order to show the surjectivity of the canonical map  of an algebra extension $B=A^{\co H} \subseteq A$ it is enough to show that for each generator $h$ of $H$ the element $1 \otimes h$ is in the image of $\chi$,  \cite{Sch00}. 
Indeed if $1 \otimes h =\chi(\sum_i a_i \otimes_B a'_i)$ and 
$1 \otimes k =\chi(\sum_j c_j \otimes_B c'_j)$ for $h,k \in H$ generic, then basic properties of the canonical map give
$  \chi (\sum_{i,j} a_i c_j \otimes_B c'_j a'_i )=1 \otimes kh $, that is $1 \otimes kh$ is in the image of $\chi$.   Being $\chi$ left $A$-linear, this implies its
surjectivity.
 
Let $\chi_\ell: A_\ell \otimes_{B_\ell} A_\ell \rightarrow A_\ell \otimes H$ be the canonical map of the algebra extension $B_\ell \subset A_\ell$ defined as in \eqref{can-map}.  We recall from \eqref{coac-n} and \eqref{coac-n2} the expression
$$
 \delta_{\ell} (\T_{i 1} ) =\T_{i 1}  \otimes p_{11}  \, ,\quad
 \delta_{\ell} (\T_{\ell 1}^{- 1}) =\T_{\ell 1}^{- 1} \otimes p_{11}^{- 1} \, ,\quad
 \delta_{\ell} ( \T_{i r}) = \sum_{k=1}^n \T_{i k} \otimes p_{k r}\,   ,\quad
  \delta_{\ell} (D^{-1}) = D^{-1}   \otimes \tilde{D}^{-1}   
$$
 for the coaction 
 $\delta_\ell: A_\ell \to A_\ell \otimes H$ on the algebra generators, for all 
 $i=1, \dots ,n$ and $ r=2, \dots ,n$. Then, from the left $A_\ell$-linearity of  the canonical map, 
 \beq\label{mappa-can-inv}
 \chi_\ell(\T_{\ell 1}^{- 1} \otimes_{B_\ell} \T_{\ell 1})= 1 \otimes p_{11}
 \, , \quad  
 \chi_\ell \Big( \sum_{k=1}^n S(\T_{i k}) \otimes_{B_\ell} \T_{k r} \Big)= 1 \otimes p_{ir}
 \, , \quad  
  \chi_\ell(D \otimes_{B_\ell} D^{-1})= 1 \otimes \tilde{D}^{-1}
   \, ,
  \eeq
for  $S$  the antipode of the Hopf algebra $A=\widetilde{\cO}_q(\rGL(n))$, with $A \subset A_\ell= A[\T_{\ell 1}^{-1}]$.   This shows the surjectivity of the canonical map $\chi_\ell$.  
 \end{proof}

For $\ell=1 $, the algebra extension $B_1 \subset A_1$ is trivial:
  \begin{lemma}
  For $\ell=1$, the map $ \gamma_1 :  H \to A_1$ defined on the algebra generators by 
    \beq\label{gamma1}
 {\gamma}_1: ~\tilde{D}^{- 1} \mapsto {D}^{- 1} \; , \quad
  p_{1j}  \mapsto \T_{1 j}   \; , \quad
   p_{r s}\mapsto \gu{1}_{r s} \; ,  \quad   j=1 , \dots , n, \; r , s=2, \dots, n ,
 \eeq 
 and extended to the whole of $H$ as an algebra map
 is a cleaving map for the 
 algebra extension $B_1 \subset A_1$,  thus trivial.
 \end{lemma}
 \begin{proof}
First we show that  the map $\gamma_1$ can be extended consistently  
 to the whole of $H$ as an algebra map.  
  From
 \eqref{commGL-mu} and \eqref{comm-rel-nu}
 \begin{align}\label{comm-an}
& \T_{1 k} \T_{1 j} =   \mu_{j k} \, \T_{1 j}  \T_{1 k}  \; , \quad
& \T_{1 k} \gu{1}_{r s} =  \mu_{1 r} \mu_{s k}  \gu{1}_{r s} \T_{1 k }  \; ,  
  \qquad
 \gu{1}_{r s} \gu{1}_{p t}= \mu_{ r p} \mu_{ts} \gu{1}_{p t}\gu{1}_{r s}
\end{align}
 moreover, from \eqref{comm-det-one} and  \eqref{comm-rel-nu2}
$$
\T_{1 k} D^{- 1} = \prod\limits_{\alpha=1}^{n} (\mu_{\alpha k} \mu_{ 1 \alpha})^{- 1} D^{-1} \T_{1 k} 
\qquad
 \gu{\ell}_{r s} D^{-1} =  \prod\limits_{\alpha=1}^{n}\big( \mu_{\alpha s} \mu_{ r\alpha}\big)^{-1} D^{- 1} \gu{\ell}_{rs} \; .
$$
Comparing these commutation relations with those among the generators $p_{ik}, \tilde{D}^{-1}$ of $H$ in \eqref{comm-p}, we see that the map $\gamma_1$ can be extended consistently  
 to the whole of $H$ as an algebra map. 
\\
It is immediate to see, comparing expressions \eqref{coac-n} and \eqref{coact-nu} for the coactions,  that the map ${\gamma}_1: H \to A_1$ is also an 
$H$-comodule morphism. 
Finally, being an algebra map, it is convolution invertible with inverse  given by $\gamma_1 \circ S$, for $S$ the antipode of $H$.
\end{proof}

\medskip
We now show that this family of $H$-Galois extensions forms a quantum principal bundle
on the quantum ringed  space $(\bP^{n-1}(\C), \cO^q_{\bP^{n-1}(\C)})$ for the 
$\{U_\ell\}$ cover as defined above.

\begin{proposition}\label{prop:sheafF}
Let the notation be as above. Then
\begin{enumerate}
\item The assignment
$\cF: U_\ell \mapsto A_\ell$, for $A_\ell$ the algebras defined in \eqref{Aell}
gives a sheaf $\cF$ of $H$-comodule algebras, for the topology generated by the covering $\{U_\ell\}$.

\item The sheaf $\cF$  is a  quantum   principal bundle (in the sense of Definition \ref{qpb-def})
over the quantum ringed space $(\bP^{n-1}(\C), \cO^q_{\bP^{n-1}(\C)})$ with respect to the covering
$\{U_\ell\}$, $\ell=1 , \dots , n$.
\end{enumerate}
\end{proposition}

\begin{proof}
(1). On the intersection of two open sets $U_\ell$ and $U_m$
with $\ell <m$ we set 
$$
\cF (U_\ell \cap U_m):= A_\ell [a_{m 1}^{-1}] =: A_{\ell m} \,
$$
and define restriction morphisms 
\beq\label{restr-maps}
\rho_{\ell, \ell m }: A_{\ell} \hookrightarrow A_{\ell m}
\qquad \rho_{m, \ell m}: A_{m} \hookrightarrow A_m [a_{\ell 1}^{-1}] \stackrel{\Psi_{m  \ell}} \longrightarrow  A_\ell [a_{m 1 }^{-1}]
\eeq
given respectively by the algebra inclusion and by the composition of the algebra inclusion with the isomorphism  
$\Psi_{m \ell}: 
A_m[\T_{\ell 1}^{-1}]    \stackrel{\simeq}{\rightarrow} A_{\ell 1}[\T_{m 1}^{-1}]$
given in \eqref{psi-lm}. Similarly for multiple intersections. We leave to the reader the
checks involved in showing that these information define a unique sheaf for the
topology generated by $\{U_\ell\}$ (see \cite[Ch. 1]{eh}  for more details), together with
the checks for the claims regarding the $H$-comodule structure.

\medskip
(2).
We have to show that  $\cF (U_\ell)^{\co H}   \simeq \cO^q_{\bP^{n-1}(\C)}(U_\ell)$, where $\cF (U_\ell)^{\co H} = B_\ell$
is the algebra generated by the elements
$  \gb{\ell}_j$ in \eqref{coinv-gln}. This is a straightforward calculation.
The $H$-Galois property was shown in Prop. \ref{prop:HG}.
\end{proof}

\subsection{Quantum reduction of $\cF$}\label{sec:qrF}
We next  construct a reduction of the quantum principal bundle $\cF$. We define the Hopf algebra  
 \beq\label{H0}
 H_0:=\widetilde{\cO}_q(P) \slash (p_{1 s}, s =2, \dots , n) \,  =\widetilde{\cO}_q(\rGL(n))/(\T_{s 1}, \T_{1 s}, s =2, \dots , n) \, ;
 \eeq 
 we denote $pr_0: H \to H_0$  the quotient map 
and mantain the notation $p_{ij}$ for the images of the generators of $H$ in the quotient Hopf algebra.
We define a
 sheaf  $\cF_0$  of $H_0$-comodule algebras over the quantum ringed space $(\bP^{n-1}(\C), \cO^q_{\bP^{n-1}(\C)})$: on the open sets $\{U_\ell \}$, $\ell =1 , \dots , n$  we set
\beq\label{F0-GLn}
\cF_0(U_\ell):=   
A_\ell  \slash (\gn{\ell}_{1s}, s=2, \dots , n)=: A^0_\ell ~, \quad
\eeq
and on the intersection of two open sets $U_\ell$ and $U_m$ with $\ell< m$ we set 
$$\cF_0 (U_\ell \cap U_m):=  
A_{\ell m} \slash (\gn{m}_{1 s}, s=2, \dots , n) =: A^0_{\ell m} \, .$$

Since $\Psi_{m \ell}(\gn{m}_{1s})=\gn{\ell}_{1 s}$, see \eqref{psi-lm2},
the isomorphism  
$\Psi_{m \ell}$ descends to a well-defined isomorphism  
$\Psi^0_{m \ell}: A^0_m[\T_{\ell 1}^{-1}]    \stackrel{\simeq}{\rightarrow} A^0_{\ell}[\T_{m 1}^{-1}]
= A^0_{\ell m}$
and we thus have
restriction morphisms 
$$\rho^0_{\ell, \ell m}: A^0_{\ell} \hookrightarrow A^0_{\ell m}
\qquad \rho^0_{m, \ell m}: A^0_{m} \hookrightarrow A^0_m [a_{\ell 1}^{-1}] \stackrel{\Psi_{m \ell}} \longrightarrow  A^0_{\ell m}.$$ The sheaf $\cF^0$ is similarly defined on multiple intersections.

For each $\ell=1, \dots , n$, the $H$-comodule algebra $A_\ell=\cF(U_\ell)$ is also an $H_0$-comodule algebra with respect to the 
 induced right coaction $(id \otimes pr_0) \circ \delta_\ell$. As from \eqref{coact-nu}, $(\gn{\ell}_{1s}, s=2, \dots , n)$ is an
 $H_0$-subcomodule of $A_\ell$. Then the coaction $(id \otimes pr_0) \circ \delta_\ell: \cF(U_\ell)=A_\ell \to A_\ell \otimes H_0$ descends to 
a well-defined coaction on  
$\cF_0(U_\ell)$ and  $\cF_0$ is
a sheaf    
of $H_0$-comodule algebras:

 \begin{proposition}
The sheaf $\cF_0$ is a  quantum principal bundle (in the sense of Definition \ref{qpb-def}) over the quantum ringed space $(\bP^{n-1}(\C), \cO^q_{\bP^{n-1}(\C)})$ with respect to the covering $\{U_\ell \}$, $\ell =1 , \dots , n$. 
\end{proposition}
\begin{proof}
By comparisons of \eqref{comm-b} and \eqref{sheaf-pnc} it is immediate to see that    
$$
\cF_0(U_\ell)^{\co H_0}=\cF(U_\ell)^{\co H} \simeq \cO^q_{\bP^{n-1}(\C)} (U_\ell).
$$
The proof that $\cF_0(U_\ell)$ is a principal  $H_0$-comodule algebra is analogue to that of Proposition \ref{prop:HG} and thus omitted. 
\end{proof}
\begin{proposition}\label{prop:ex-gln}
The quantum principal bundle $\cF_0$ is a quantum algebraic reduction of $\cF$  (in the sense of Definition \ref{qred-def}). 
\end{proposition}
\begin{proof}
The surjective $H_0$-comodule maps
$\varphi_{U}: \cF(U_j) \lra \cF_0(U_j)$, compatible with the restriction maps $\rho_{-,-}$ and  $\rho^0_{-,-}$ are given by the 
quotient maps
$
\varphi_{U}: \cF(U_j) \lra \cF_0(U_j) = \cF(U_j) \slash (\gn{j}_{1s}, s=2, \dots , n).
$ Cf. \eqref{psi-lm2}.
\end{proof}
\medskip

\begin{remark}  
The construction  we have  presented here  works more in general for a quadratic algebra defined by $n^2$ generators which have commutation relations as in \eqref{commGL-mu}, with $\mu_{i j}$ parameters of deformations satisfying the properties $\mu_{i j} \mu_{j i}=1$, for all $i,j=1, \dots , n$ and \eqref{prop2}. Indeed the computations only use this latter property of the parameters and not their explicit expression \eqref{mu-coeff}.

In particular it works for 
the deformation $\mathcal{O}_\theta (\rGL(n))$ of the algebra of coordinates on $\rGL(n)$ 
 introduced  by Connes and  Dubois--Violette in \cite{cdv}, $n >2$. This latter 
has generators $a_{i j}$, $i,j=1 , \dots , n $, which satisfy commutation relations that, formally, are written as in \eqref{commGL-mu}, with 
$\mu_{ij}:=exp(-2 i \pi \theta_{ij})$ for $\theta$ an antisymmetric matrix of real deformation parameters. 
\beq\label{thetagen}
\theta_{ij}=-\theta_{ji}=-\theta_{i' j} \,   , \quad i,j=1, \dots n .
\eeq
 with $j':=m+j$ for $j \leq m$ and $n=2m$ or $n=2m+1$ and,  in this last case, $j'=j$ for $j=2m+1$.
In this case 
$
\prod\limits_{\alpha=1}^{n} \mu_{\alpha k}= \prod\limits_{\alpha=1}^{m} \mu_{\alpha k} \prod\limits_{\alpha=m+1}^{2m} \mu_{\alpha k} = \prod\limits_{\alpha=1}^{m} \mu_{\alpha k}   \mu_{\alpha' k} =1
$
and thus the determinant  is central.  
It is only when requiring  the validity of Proposition \ref{prop:b1=bl}, that one has to impose the further restriction \eqref{prop2} on the parameters, reducing the number of parameters of deformation to just one. 
\qed
\end{remark}

\subsection{The $n=2$ case: quantum reduction from   $\widetilde{\cO}_q(\rGL(2))$.}
We detail here the case $n=2$, providing a noncommutative version of the sheaf reduction in Example \ref{example-red}. 
In this low-dimensional case the ringed  base  space $(\bP^1(\C), \cO_{\bP^1(\C)})$ is classical
and the sheaves $\cF$ and $\cF_0$ give examples of quantum locally cleft principal bundles.

 \medskip
 
Let  $A=\widetilde{\cO}_q(\rGL(2))$. Accordingly to \eqref{multi2},  it has generators $a$, $b$, $c$, $d$ and  $ D^{-1} $
which satisfy commutation relations
\begin{align} 
a b = q^{-1} \, b a \; , \quad  
a c  = q \, c a  \; , \quad 
c d = q^{-1}  \, d c \; , \quad 
b d = q \, d b     \; , \quad 
b c  = q^2  \, c b    \; , \quad 
a d =   d a   \, 
\end{align}
together with 
$$
a D^{\pm 1} = D^{\pm 1}  a \; , \quad 
b D^{\pm 1}  = q^{\pm 2}  D^{\pm 1}  b \; ,  \quad
c D^{\pm 1}  =q^{\mp 2}  D^{\pm 1} c \; , \quad 
d D^{\pm 1}  = D^{\pm 1}  d 
$$
for $ D = ad - q^{-1} bc$ the quantum determinant  and $D^{-1}$ its inverse.
Let   $H:= \widetilde{\cO}_q(P):=\widetilde{\cO}_q(\rGL(2))/(c)$ be the Hopf algebra obtained as quotient of $A$  for the Hopf algebra ideal generated by $c$. The induced  coaction  of $H$ on $\widetilde{\cO}_q(\rGL(2))$  is given on the  generators by
\beq\label{act-P2}
\delta: 
\begin{pmatrix} a & b \\ c & d 
\end{pmatrix} \mapsto \begin{pmatrix} a & b \\ c & d 
\end{pmatrix} \otimes \begin{pmatrix} t & n \\ 0 & s 
\end{pmatrix} \; , \quad D^{-1} \mapsto D^{-1}  \otimes \tilde{D}^{-1}
\eeq
and extended to the whole algebra as an algebra homorphism.
Here $t,n,s,\tilde{D}^{-1} \in \widetilde{\cO}_q(P)$, with $t(s \tilde{D}^{-1})= (s \tilde{D}^{-1})t=1$,  are respectively the images of the algebra generators $a,b,c, D^{-1} \in \widetilde{\cO}_q(\rGL(2))$ under the  quotient map $A \to H$. \\

We define algebra Ore extensions
\begin{align*}
 A_1 &:=A[a^{-1}]  
  \\ 
&  ~
\simeq  \C_q[a^{\pm 1},x:=a^{-1}b,c, D^{\pm 1}] \slash \big(
a^{\pm 1} x -q^{\mp 1}  x a^{\pm 1}  \, , \; 
c x - q^{-1} x c  \, , \;
a^{\pm 1} c - q^{\pm 1} ca^{\pm 1}  \, , 
\\  & \hspace{5,5cm} 
a^{\mp 1} a^{\pm 1}-1   \, , \;
a  D^{\pm 1} = D^{\pm 1}  a \, , \quad 
a^{-1}  D^{\pm 1} = D^{\pm 1}  a^{-1} \, , 
\\ & \hspace{5,5cm}
x D^{\pm 1}  = q^{\pm 2}  D^{\pm 1}  x \, ,  \quad
c D^{\pm 1}  =q^{\mp 2}  D^{\pm 1} c \big)
\end{align*}

\begin{align*}
A_2 &:=A[c^{-1}] 
\nn   \\ 
&  ~
\simeq \C_q[a,y:=c^{-1}d,c^{\pm 1}, D^{\pm 1}] \slash \big(
a c^{\pm 1} -q^{\pm 1}  c^{\pm 1}  a \, , \; 
c^{\pm 1} y -q^{\mp 1}  y c^{\pm 1}  \, , \;
ay- q^{-1} ya   \, , 
\\ & \hspace{5,5cm}
c^{\mp 1} c^{\pm 1}-1 \, , \;
a  D^{\pm 1} = D^{\pm 1}  a \, , \quad 
c D^{\pm 1}   =q^{\mp 2}  D^{\pm 1}  c  \, ,
\\ & \hspace{5,5cm}
c^{-1}  D^{\pm 1}   =q^{\pm 2}  D^{\pm 1}  c^{-1} \, , \;
y D^{\pm 1}  = q^{\pm 2} D^{\pm 1}  y ) \big)
\end{align*}
and
\begin{align*}
A_{12}&:=A_1[c^{-1}]\simeq A_2[a^{-1}] 
\end{align*}
with algebra isomorphism (cf. \eqref{change-coord})
\beq\label{iso-Psi}
\Psi:A_1[c^{-1}]\stackrel{\simeq}{\longrightarrow} A_2[a^{-1}]  \, , \qquad  
a^{\pm 1}   \mapsto  a^{\pm 1}  \, , \quad
c^{\pm 1} \mapsto    c^{\pm 1}  \, , \quad
  x  \mapsto    y \, , \quad
 D^{\pm 1} \mapsto   D^{\pm 1} \, .
 \eeq
The coaction  \eqref{act-P2}  extends, uniquely as an algebra map, to coactions
$\delta_i:A_i \lra A_i \otimes  \widetilde{\cO}_q(P)$ with
$$
\delta_1 (a^{-1})=a^{-1}\otimes t^{-1}~ , \qquad
\delta_2 (c^{-1})=c^{-1}\otimes t^{-1}~ , 
$$
for $t^{-1}:= s \tilde{D}^{-1}$, or
\beq\label{coac-xy}
\delta_1 (x)=  1 \otimes  t^{-1} n + x \otimes  t^{-1} s \; , \quad \delta_2 (y)= 1 \otimes  t^{-1} n + y \otimes  t^{-1} s~.
\eeq
The corresponding subalgebras  $B_i:=A_i^{\co \, \widetilde{\cO}_q(P)}$ of
coinvariant elements are  
$$
B_1=\C_q[a^{-1}c] \simeq \C[z], \qquad B_2=\C_q[ac^{-1}] \simeq \C[w]\; 
$$
and $B_{12}=\C_q[ac^{-1}, c a^{-1}] \simeq \C[z^{\pm 1}]$. 
In the algebras $A_1$, respectively $A_2$, we have
\beq\label{scomp1}
\begin{pmatrix} a & b \\ c & d \end{pmatrix} 
= 
\begin{pmatrix} 1 & 0 \\ c a^{-1} & 1  \end{pmatrix} 
\begin{pmatrix} a & b \\ 0 & a^{-1} D \end{pmatrix}
= 
\begin{pmatrix} 1 & 0 \\ c a^{-1} & 1  \end{pmatrix} 
\begin{pmatrix} a & 0 \\ 0 & a^{-1} D \end{pmatrix}
\begin{pmatrix} 1 & x \\ 0 & 1 \end{pmatrix}
\eeq
\beq\label{scomp2}
\begin{pmatrix} a & b \\ c & d \end{pmatrix} 
= 
\begin{pmatrix} a c^{-1}  & 1 \\ 1 & 0  \end{pmatrix} 
\begin{pmatrix} c & d \\ 0 & -q^{-1} c^{-1} D \end{pmatrix} 
= 
\begin{pmatrix} a c^{-1}  & 1 \\ 1 & 0  \end{pmatrix} 
\begin{pmatrix} c & 0 \\ 0 & -q^{-1} c^{-1} D \end{pmatrix} 
\begin{pmatrix} 1 & y \\ 0 & 1 \end{pmatrix}.
\eeq 
The algebras $A_i$ are also $H_0$-comodule algebras for the 
 induced right coactions $(id \otimes pr_0) \circ \delta_i$ of the Hopf algebra $H_0:=H \slash (n)$, the quotient of $H$ by the Hopf algebra ideal generated by $n$, with  $pr_0: H \to H_0$ the projection map. Notice that the algebra $H_0$ is commutative with all elements which are group-like. We keep denoting $t^{\pm 1}, ~ D^{\pm 1}$ the images of the generators in $H_0$.
\medskip

We consider the topology $\mathcal{U}$ induced by the base $\{U_1, U_2,U_{12}:=U_1\cap U_2\}$ with $U_1$ and $U_2$ being the subsets defined in the Example \ref{example-red}.
We obtain a  sheaf $\cF$ of $\widetilde{\cO}_q(P)$-comodule algebras on the ringed space $(\bP^1(\C),
\cO_{\bP^1(\C)})$ by setting
\beq\label{F-SL}
\cF(\bP^1(\C))= A ~, \quad 
\cF(U_1)=A_1~, \quad 
\cF(U_2)=A_2~, \quad 
\cF(U_{12})=A_{12} 
~, 
\eeq
and $\cF(\emptyset):=\{0\}$ (the  one element algebra).  The
 restriction morphisms $\rho_{U,V}: \cF(U) \to \cF(V)$,  for $V \subseteq U \in \mathcal{U}$,  are all given by algebra inclusions $\imath$, 
except for $ \rho_{U_2,U_{12}}:= i \circ \Psi$ given by the composition with 
 the $H$-comodule algebra isomorphism $\Psi$ in \eqref{iso-Psi}.

 \begin{proposition}\label{prop:Floc-cleft}
The sheaf $\cF$ is a  quantum locally cleft  principal bundle (in the sense of Definition \ref{qpb-def}) over the ringed space $(\bP^1(\C),
\cO_{\bP^1(\C)})$ with respect to the covering $\{U_1, U_2\}$. 
\end{proposition}
\begin{proof}
We  show  that $\cF$ is locally cleft. 
  The  algebra extension $B_1 \subset A_1$  is a trivial $\widetilde{\cO}_q(P)$-Galois extension with (unital)
cleaving map which is given on the generators by 
\beq\label{cleaving1-gl2}
\gamma_1: \widetilde{\cO}_q(P)  \to A_1 \; , \qquad  
\tilde{D}^{\pm 1} \mapsto D^{\pm 1} \, , \quad
t^{\pm 1} \mapsto a^{\pm 1} \, , \quad  
n \mapsto a x =b  \;
\eeq
or (cf.  \eqref{scomp1}) in matrix form by
$
\gamma_1 :  \begin{pmatrix} t & n \\ 0 & t^{-1} \tilde{D} \end{pmatrix} 
\mapsto
\begin{pmatrix} a & b \\ 0 & a^{-1} D \end{pmatrix}
$.
Indeed, an easy computation on the generators shows that $\gamma_1$ can consistently be extended as an algebra map to the whole of $\widetilde{\cO}_q(P)$. It
 is a convolution invertible  $\widetilde{\cO}_q(P)$-comodule map with inverse  $\bar{\gamma}_1: \widetilde{\cO}_q(P)  \to A_1$ given by the algebra map  
 $\bar{\gamma}_1 := \gamma_1 \circ S$ for $S$ the antipode of $P$.

As for the  other chart, the computation is more involved since the cleaving map is not an algebra map. 
The extension $B_2 \subset A_2$  is cleft   with (unital)
cleaving map  which is given on the vector space basis  
$\{\tilde{D}^p t^m  n^{r}  , ~ m,p \in   \mathbb{Z}, r \in   \mathbb{N}\}$ of $\widetilde{\cO}_q(P)$ by
\beq\label{cleaving2-gl2}
\gamma_2: \widetilde{\cO}_q(P)  \to A_2 \; , \qquad  
\tilde{D}^p t^m   n^{r} \mapsto  (-1)^p q^p  D^p c^{m} d^r  
\eeq
and on the algebra generators reads 
\beq
\gamma_2 :  \begin{pmatrix} t & n \\ 0 & t^{-1} \tilde{D} \end{pmatrix} 
\mapsto
\begin{pmatrix} c & d \\ 0 & -q^{-1} c^{-1} D \end{pmatrix}
\eeq
(cf. \eqref{scomp2}). We postpone to Appendix  \ref{app:A} the proof that $\gamma_2$ is a convolution invertible $\widetilde{\cO}_q(P)$-comodule map.
\end{proof}

We  define another sheaf  $\cF_0$  over  $\bP^1(\C)$
 by setting 
\begin{align}\label{F0-SL}
&\cF_0(U_1):=\C_q[a,c,a^{-1}, D^{\pm 1}] \subset \cF(U_1) ~, \quad
&&\cF_0(U_2):=\C_q[a,c, c^{-1}, D^{\pm 1}] \subset \cF(U_2) ~, \nonumber \\ 
&\cF_0(U_{12}):= \C_q[a,c, a^{-1},c^{-1}, D^{\pm 1}] \subset \cF(U_{12}) ~, \quad
&&\cF_0(\bP^1(\C))=\widetilde{\cO}_q(\rGL(2)) \; 
\end{align}
with restriction morphisms $\rho^0_{U,V}$ given by algebra inclusions.
Observe that
$$
 \cF_0(U_1) \simeq \cF(U_1) \slash (x) \, , \quad
 \cF_0(U_2)\simeq  \cF(U_2) \slash (y)\, 
$$
with $(x)$, respectively $(y)$,  $H_0$-subcomodules of $\cF(U_1)$, respectively $\cF(U_2)$. Indeed, from \eqref{coac-xy}, one has 
$(id \otimes pr_0) \circ \delta_i(x) =   x \otimes  t^{-1} s$ and $(id \otimes pr_0) \circ \delta_i(y)=   y \otimes  t^{-1} s$. Then,  the 
maps $(\varphi_i \otimes id)\circ (id \otimes pr_0) \circ \delta_i$, for
\beq\label{mappe-phi}
\varphi_1:\cF(U_1) \lra \cF_0(U_1)  \simeq \cF(U_1) \slash (x)
\; , \quad
\varphi_2:\cF(U_2) \lra \cF_0(U_2)  \simeq \cF(U_2) \slash (y)
\eeq
the quotient maps,
descend to well-defined coactions    
$\delta_i^0: \cF_0(U_i) \to  \cF_0(U_i)  \otimes H_0 $ and  $\cF_0$ is
a sheaf    
of $H_0$-comodule algebras.  Analogously for $U_{12}$, with
\beq\label{mappe-phi2}
\varphi_{12}:\cF(U_{12}) \lra \cF_0(U_{12})  \simeq \cF(U_{12}) \slash (x) ~.
\eeq

 \begin{proposition}
The sheaf $\cF_0$ is a  locally cleft   quantum principal bundle (in the sense of Definition \ref{qpb-def}) over the ringed space $(\bP^1(\C),
\cO_{\bP^1(\C)})$ with respect to the covering $\{U_1, U_2\}$. 
\end{proposition}
\begin{proof}
It is easy to see that    
$$
\cF_0(U_1)^{\co H_0}=\C_q[a^{-1}c] \simeq \C[z], \qquad \cF_0(U_2)^{\co H_0}=\C_q[ac^{-1}] \simeq \C[w]
$$
are the subalgebras of  coinvariant elements.
The  algebra extension $\cF_0(U_1)^{\co H_0} \subset \cF_0(U_1)$ is trivial
with (unital) cleaving map which is given by the algebra  map
$$
\gamma^0_1: H_0 \to \cF_0(U_1) \; , \qquad  
\tilde{D}^{\pm 1} \mapsto D^{\pm 1} \, , \quad
t^{\pm 1} \mapsto a^{\pm 1}  $$
or in matrix form (cf.  \eqref{scomp1})
\beq
\gamma_1^0 :  \begin{pmatrix} t & n \\ 0 & t^{-1} \tilde{D} \end{pmatrix} 
\mapsto
\begin{pmatrix} a & 0 \\ 0 & a^{-1} D \end{pmatrix}
\eeq
Recalling that the algebra $H_0$ is commutative with all elements which are group-like, it is immediate to see that $\gamma^0_1$ 
 is a convolution invertible  $H_0$-comodule map.

The  algebra extension $\cF_0(U_2)^{\co H_0} \subset \cF_0(U_2)$ is cleft
with (unital) cleaving map which is given on the vector space basis  
$\{ \tilde{D}^p t^m , ~ m,p \in   \mathbb{Z} \}$ of $H_0$  by
\beq
\gamma^0_2: H_0 \to \cF_0(U_2) \; , \qquad  
\tilde{D}^p t^m   \mapsto (-q)^p D^p c^{m}  
\eeq
and reading   (cf. \eqref{scomp2})
\beq
\gamma_2^0 :  \begin{pmatrix} t & n \\ 0 & t^{-1} \tilde{D} \end{pmatrix} 
\mapsto
\begin{pmatrix} c & 0 \\ 0 & -q^{-1} c^{-1} D \end{pmatrix}
\eeq
on the algebra generators.
 It is  an $H_0$-comodule map
$$
(\gamma_2^0 \otimes id) \Delta (\tilde{D}^p t^m)= 
(-q)^p D^p c^{m} \otimes \tilde{D}^p t^m
=
(-q)^p \delta_2^0 (D^p c^{m} )
=
\delta_2^0 \big( \gamma_2^0(\tilde{D}^p t^m) \big) 
$$
with convolution inverse 
$\bar{\gamma}_2^0:  \,   \tilde{D}^p t^m   \mapsto (-q)^{-p} c^{-m} D^{-p} $.
\end{proof}

\begin{proposition}
The quantum principal bundle $\cF_0$ is a quantum algebraic reduction of $\cF$  (in the sense of Definition \ref{qred-def}). 
\end{proposition}
\begin{proof}
We show there exist surjective $\widetilde{\cO}_q(T)$-comodule maps
$\varphi_U: \cF(U) \lra \cF_0(U)$,
for all $U \in \mathcal{U}$, that are compatible with the restriction maps $\rho_{-,-}$ and  $\rho^0_{-,-}$  of the sheaves $\cF$ and  $\cF_0$, respectively. That is, 
$$
 \varphi_V \circ \rho_{U,V} =\rho^0_{U,V} \circ \varphi_U
$$
for all $V \subseteq U \in \mathcal{U}$,
for $\Psi$ the $H$-comodule algebra isomorphism in \eqref{iso-Psi}.

For $U=\bP^1(\C)$ we take the identity map $id: \cF(\bP^1(\C))  \to \cF_0(\bP^1(\C))$.
By construction, the quotient maps
$\varphi_1, ~\varphi_2$ in \eqref{mappe-phi}  and $\varphi_{12}$ in \eqref{mappe-phi2}  
are surjective $H_0$-comodule algebra maps, 
$
 (\varphi_{i} \otimes id) \circ (id \otimes pr_0) \circ  \delta  =   \delta_0 \circ \varphi_{i} 
$ and make the following diagrams commutative
$$
\xymatrix{
  \cF(U_1) \ar@{->}_{\rho_{U_1,U_{12}}}[d]  \ar@{->>}^{\varphi_1}[r]&  \cF_0(U_1) \ar@{->}^{\rho^0_{U_1,U_{12}}}[d]  \\
 \cF(U_{12}) \ar@{->>}^{\varphi_{12}}[r]   & \cF_0(U_{12})
}
\qquad  \quad
\xymatrix{
  \cF(U_2) \ar@{->}_{\rho_{U_2,U_{12}= i \circ \Psi}}[d]  \ar@{->>}^{\varphi_2}[r]&  \cF_0(U_2) \ar@{->}^{\rho^0_{U_2,U_{12}}}[d]  \\
 \cF(U_{12}) \ar@{->>}^{\varphi_{12}}[r]   & \cF_0(U_{12})
} \, .
$$
Moreover,   $\varphi_{i}(\cF(U_i)^{\co H})=\cF(U_i)^{\co H_0}$ for the two open sets $U_1,~U_2$.  
 \end{proof}
 
\section{Quantum sheaf algebraic Reductions}\label{red-sec}

In this section we provide a sufficient condition for the existence of reductions of quantum principal bundles,  yielding them directly,  based
on a sheaf theoretic reinterpretation of the result in \cite{gunther}.

\subsection{Quantum affine algebraic Reductions}\label{sec:gunther}

We recall some results about affine reductions, in particular
the Hopf-Galois reduction Theorem \ref{hg-thm}. The main reference is
\cite{gunther}, see also \cite{br-red, hajac-red}.  
\medskip

As before, let $H$ be a Hopf algebra  with bijective antipode and $H_0:=H/J$ its quotient by a 
 Hopf ideal $J$, with $pr_0:H \lra H_0=H/J$ denoting the quotient map. Let 
 $$
{}^{\co H_0}H=\{ k \in H ~|~ pr_0(\one{k}) \otimes \two{k}=
1 \otimes k \in H_0 \otimes H \}
$$
be
the subalgebra of coinvariants of $H$ with respect to the
left coaction $(pr_0 \otimes \mathrm{id}) \circ \Delta$ of $H_0$. We assume ${}^{\co H_0}H \subset H$ to be a
principal $H_0$-Galois extension (and hence $H$ is coflat as $H_0$--comodule).

The Hopf-Galois Reduction Theorem 
is a noncommutative algebraic counterpart of Prop.
\ref{thm-class-red}. It was proven
in \cite[Thm. 4]{gunther}, as a generalization of \cite[Thm. 3.2]{Sch}
for Galois objects.
It gives a   correspondence between algebraic $H_0$-reductions
of a $H$--Galois
extension $B:=A^{\co H} \subset A$ (in the sense of Def. \ref{affine-red}) and 
certain algebra maps from the subalgebra of coinvariants ${}^{\co H_0}H$ to
the centralizer
subalgebra of $B$ in $A$.
In order to state the result, let us introduce some notation and
preliminary results.

Recall that, for  any $H$--Galois extension $B \subset A$,
the centralizer subalgebra of $B$ in $A$,
\beq\label{centralizer}
Z_A(B):= \{x \in A ~ | ~ x b=b x , ~ \forall b \in B\}
\eeq
is a right $H$-module algebra and $H$-comodule algebra.  The action is given by  the so-called Miyashita--Ulbrich action
(see \cite[\S 3]{dt}):
\beq\label{MU-action}
\triangleleft_{MU} : Z_A(B) \otimes H \to Z_A(B) \; , \quad  x \otimes h \mapsto x
\triangleleft_{MU} h := \tone{h} x \ttwo{h}
\eeq
where     $\tone{h}$ and  $\ttwo{h}$ denote the components of the translation map $\tau$ of
the extension $B \subset A$, $\tau(h):= \tone{h} \otimes_B \ttwo{h} \in A \otimes_B A$
(see \S\ref{qpb-sec}, (\ref{transl-map})).
Being  $\tau$ valued on the balanced tensor product, the action \eqref{MU-action} 
is indeed well-defined only on the centralizer subalgebra $Z_A(B) \subseteq A$.
The right coaction is given by the restriction to $Z_A(B)$ of the principal $H$-coaction on $A$,
\beq\label{princ-co}
\delta: x \mapsto \zero{x} \otimes \one{x}.
\eeq 
The subalgebra ${}^{\co H_0}H$  of coinvariants of $H$ with respect to the
left coaction of $H_0$ is a right $H$ module and  comodule algebra too. It has right adjoint action
$$
\triangleleft : {}^{\co H_0}H \otimes H \to {}^{\co H_0}H \; ,
\quad k  \otimes h \mapsto k \triangleleft h := S(\one{h}) k \two{h}
$$ 
and coaction given by the restriction of the coproduct of $H$
to its subalgebra ${}^{\co H_0}H$, indeed valued in ${}^{\co H_0}H \otimes H$.
\begin{observation}
We recall that both $Z_A(B)$ and ${}^{\co H_0}H$ are Yetter--Drinfeld module algebras over $H$. 
The  $H$-module (\ref{MU-action}) and 
$H$-comodule (\ref{princ-co}) structures on $Z_A(B)$ are shown to satisfy the Yetter--Drinfeld condition
$$
\delta (x \triangleleft_{MU} h)= \zero{x} \triangleleft_{MU} \two{h} \otimes S(\one{h})\one{x} \three{h}
$$
 by using the property 
$
\one{\tuno{h}} \otimes \zero{\tuno{h}}  \otimes_B \zero{\tdue{h}} \otimes \one{\tdue{h}} = S(\one{h})  \otimes \tuno{\two{h}} \otimes_B \tdue{\two{h}} \otimes 
\three{h}$
of the translation map, obtained by
combining properties \eqref{p1} and \eqref{p2}.
The Yetter--Drinfeld compatibility condition for the subalgebra ${}^{\co H_0}H$ is an easy check. 
\end{observation}

We denote by
$\mathrm{Alg}_H^H({}^{\co H_0}H, Z_A(B))$ 
the space  of maps $f:{}^{\co H_0}H \to  Z_A(B)$ that are  $H$-module and comodule algebra maps.
The $H$-module map condition reads
\beq\label{MU-condition}
f(S(\one{h})k \two{h} )= \tone{h} f(k) \ttwo{h} \; , \quad
\forall h \in H, ~k \in {}^{\co H_0}H  
\eeq
where as before   $\tone{h} \otimes_B \ttwo{h} =\tau(h)$ is the translation map $\tau$ of
the extension $B \subset A$, while the $H$-comodule map condition is simply
\beq\label{com-cod}
\zero{f(h)} \otimes \one{f(h)} = f(\one{h}) \otimes \two{h} \, .
\eeq

For the Hopf-Galois Reduction Theorem \cite[Thm 4]{gunther}, functions in $\mathrm{Alg}_H^H({}^{\co H_0}H, Z_A(B))$  are in  correspondence with  elements in  the set ${}_B \mathrm{Red}^{H_0} (A)$ of 
algebraic $H_0$-reductions  of the $H$--Galois
extension $B \subset A$. 
In particular, given
a function  $f \in \mathrm{Alg}_H^H({}^{\co H_0}H, Z_B(A))$,
the set $I^+:=A f({}^{\co H_0}H \cap ker(\varepsilon))$ is an ideal in $A$, by 
condition \eqref{MU-condition},  and an $H$-subcomodule. Moreover  
$A \slash  I^+$ is an $H_0$-Galois extension of $B$ and a reduction of $B \subset A$. 
This gives the map
\begin{eqnarray}\label{hg-thm}
\mathrm{Alg}_H^H({}^{\co H_0}H, Z_B(A)) &\longrightarrow& {}_B \mathrm{Red}^{H_0} (A) \nonumber \\
f &\longmapsto& A/A \, f({}^{\co H_0}H \cap ker(\varepsilon)) .
\end{eqnarray}

We conclude this section with a technical lemma we shall
need in the sequel.
\begin{lemma}\label{lem:MUprod}
Let $f:{}^{\co H_0}H \to  Z_A(B)$ a linear map. 
Suppose  condition  \eqref{MU-condition} holds for two elements $h,h'$ in $H$, for all
$k \in {}^{\co H_0}H $.  Then condition  \eqref{MU-condition}  is also satisfied by the
element $hh'$, for all $k \in {}^{\co H_0}H $. 
\end{lemma} 
\begin{proof}
For the element $hh'$,   we  compute
\begin{align*}
f\big(S(\one{(hh')})k \two{(hh')} \big)
&= f\Big(S(\one{h'}) ~\big( S(\one{h})k \two{h}   \big) ~\two{h'} \Big)
\\
&
= \tone{h'} f\big( S(\one{h})k \two{h}   \big) ~\ttwo{h'} 
\\
&
= \tone{h'}\tone{h} f (k) \ttwo{h}   \ttwo{h'} 
\\
&
= \tone{(hh')} f(k) \ttwo{(hh')} 
\end{align*}
thus proving  \eqref{MU-condition}.
The last equality holds by property (\ref{tau-st}) of $\tau$.
\end{proof}

\subsection{Quantum sheaf algebraic reductions}
Next, we want to take advantage of the above correspondence
for the affine setting to obtain a result for the sheaf theoretic
definition of reduction. This will allow us to consider examples
where the global coordinate rings of the geometric objects are
not available.\\

Let $\cF$ be a quantum $H$-principal bundle over the
quantum ringed space $(M, \cO_M)$ with respect to an open covering
$\{U_i\}$ as in Definition \ref{qpb-def}.
For each $i$, let $Z_{\cF(U_i)}\big(\cO_M(U_i)\big)$ be the centralizer
of $\cF(U_i)^{\\co H} =\cO_M(U_i)$  in $\cF(U_i)$ with right action defined as in \eqref{MU-action}
and right $H$-coaction given by the restriction of that on $\cF(U_i)$.
{Assume, by eventually restricting $\cF$, the topology to be generated by
the open covering $\{U_i\}$, so that $\{U_I:=\cap_{i\in I} U_i\}$ is a base (see also
\cite[\S 4]{AFLW} )}  and $I$ is a multindex, $I=(i_1, \dots, i_r)$,  $1\leq i_1< \dots < i_r \leq n$. In the following we will denote for simplicity  by $\rho_{K,I}$ the restriction map from $U_K$ to $U_I$.

\begin{theorem}\label{prop:gunther-sheaf}
Let $\cF$ be a quantum $H$-principal bundle over the
quantum ringed space $(M, \cO_M)$  with respect to a finite open covering $\{U_i\}$, as above.
Let $H_0=H/J$  
as above.
Assume $\{f_i: {}^{\co H_0}H \lra Z_{\cF(U_i)}\big(\cO_M(U_i)\big)\}$
is a family of $H$-module and $H$-comodule algebra 
maps  
such that the following diagram commutes  
\beq\label{glue}
\xymatrix{
  {}^{\co H_0}H  
\ar@{->}_{f_{j}}[d]   \ar@{->}^{f_i \qquad }[r] &  Z_{\cF(U_i)}\big(\cO_M(U_i)\big) \subseteq \cF(U_i)
\ar@{->}^{\rho_{i,ij}}[d]   \\
Z_{\cF(U_j)}\big(\cO_M(U_j)\big) \subseteq   \cF(U_j)    \ar@{->}^{\qquad \rho_{j,ij}}[r]  &
\cF(U_i \cap U_j)  
}
\eeq

Then 
$\cF$ admits an algebraic reduction to $H_0$.
\end{theorem}

\begin{proof} 
Suppose we have a family of maps $f_i$ as above. We want to define
a reduction of $\cF$ as a quotient $\cF_0:=\cF/\cI$, where
$\cI$ is an ideal sheaf. We define $\cI$ by using locally the correspondence   \eqref{hg-thm} of the Hopf--Galois reduction Theorem. 
For each open set $U_I$ of the base, we set
$$
\cI(U_I):= \cF(U_I) (\rho_{i,I} \circ f_i) \big( {}^{\co H_0}H \cap \ker \varepsilon_H \big)\, .
$$
This is well-defined, that is it does not depend on the choice of the index $i$.
In fact since
$$
\rho_{i, I } = \rho_{ij ,I} \circ \rho_{i, ij }
$$
and by hypothesis \eqref{glue} holds, $\rho_{i, ij } \circ f_i = \rho_{j, ij } \circ f_j$,
we have that:
$$
\rho_{i,I} \circ f_i =  \rho_{ij ,I} \circ \rho_{i, ij } \circ f_i=
\rho_{j,I} \circ f_j  \; .
$$

\medskip
The restriction maps $\rho_{K,I}$ of the sheaf $\cF$ ensure that $\cI$
has the presheaf property on the base $\{U_I\}$. By \cite[Ch.0, \S 3.2.1]{groth},
we can extend the presheaf on the base to a presheaf on all open sets and then,
as usual, via sheafification, get an ideal sheaf, that we still denote by $\cI$.
Hence we obtain a sheaf $\cF_0:=\cF/\cI$. By applying  the
Hopf-Galois Reduction Theorem  to $\cF(U_i)/\cI(U_i)$, we see this latter is a faithfully flat $H_0$-Galois extension of
$\cO_M(U_i)$ and so is $\cF_0(U_i)$ obtained via sheafification (it is a finite inverse limit).
We leave to the readers all the checks of the properties involving the restriction maps
$\rho^0_{K,I}$ of $\cF_0$ on the base and then extended naturally to all open sets.
\end{proof}

\subsection{Example}\label{ex:rid-f}  

We return to the examples  in  \S \ref{ex-sec}  and  
show here the existence of the quantum reduction by constructing maps $f_j$ as
in Thm \ref{prop:gunther-sheaf}.
We keep the same notation used there.
\medskip

Let $H=\widetilde{\cO}_q(P):=\widetilde{\cO}_q(\rGL(n))/(\T_{s1}, s \neq 1)$ and let
$J$ be the  ideal generated by $\{ p_{1 s}, s =2, \dots , n \}$.  It is a Hopf algebra ideal of $H$ and 
$H_0=\widetilde{\cO}_q(P) \slash J$ is the Hopf algebra  in \eqref{H0}.
By standard arguments, 
$H$ is a left $H_0$-comodule algebra with coaction $\rho: H \to H_0 \otimes H$,
$h \mapsto \pi(\one{h}) \otimes \two{h}$, for $\pi: H \to H_0$  the quotient map,
$\overline{p}_{ij}:=\pi(p_{ij})$.

We notice that $H$ is a trivial $H_0$-Galois extension of $ {}^{\co H_0}H$, by writing directly 
the cleaving map $\jmath:H_0 \lra H$ and its convolution inverse $\bar{\jmath}$:
$$
\jmath(\overline{p}_{ij})=p_{ij}, \qquad \bar{\jmath}(\overline{p}_{11})=p_{11}^{-1}, \quad
\bar{\jmath}(\overline{p}_{ij})=S(p_{ij}).
$$

\medskip
The subalgebra  $ {}^{\co H_0}H$ of coinvariants is generated by the elements

$$\beta_s:=p_{11}^{-1}p_{1s}\; , \quad s=2,\dots, n$$ with commutation relations 
\beq\label{comm-beta}
\beta_s \beta_r = \mu_{r s} \beta_r \beta_s \, , \quad r,s=2,\dots, n .
\eeq
As in \S\ref{ex-sec}, we consider the quantum ringed space $(\bP^{n-1}(\C),\cO^q_{\bP^{n-1}(\C)})$ with 
open cover given by the sets $U_\ell$  in \eqref{aperti-U}, for $\ell=1 , \dots , n$, that we assume to generate the topology. Thus the intersections
of an arbitrary number of $U_\ell$'s  form a basis.  
\medskip

In Propositon \ref{prop:sheafF} we proved that the assignment
$\cF: U_\ell \mapsto A_\ell=A[a_{\ell 1}^{-1}]$, for $A_\ell$ the algebras defined in \eqref{Aell},
gives a sheaf $\cF$ of $H$-comodule algebras for the topology generated by the $\{U_\ell\}$ and moreover that it is 
a  quantum   principal bundle (in the sense of Definition \ref{qpb-def})
 with respect to the covering
$\{U_\ell\}$. 
We are now going to show the existence of a reduction of $\cF$ to  a quantum   principal $H_0$-bundle by showing 
the existence of algebra morphisms $f_\ell : {}^{\co H_0}H \to Z_{B_\ell} (A_\ell)$,  $\ell=1, \dots , n$, which satisfy the hypothesis of Thm \ref{prop:gunther-sheaf}.
\medskip

For $\ell=1, \dots , n$, we define the map $f_\ell: {}^{\co H_0}H \to  A_\ell$  on the algebra generators $\beta_k$  as 
\beq\label{mappe-f}
f_\ell: 
\beta_s \mapsto \gn{\ell}_{1s}=\T_{\ell 1}^{-1}\T_{\ell s}\, , \quad s=2,\dots, n .
\eeq
By comparing the commutation relations of the   $\beta_s$ in \eqref{comm-beta} with those of the elements $\gn{\ell}_{1s}$ in \eqref{comm-rel-nu}, we see that 
the map  $f_\ell$  can be extended to the whole algebra ${}^{\co H_0}H$ by requiring it to be  an algebra morphism.

\begin{proposition}
The algebra maps $f_\ell$ in \eqref{mappe-f} are valued in the   centralizer $Z_{B_\ell} (A_\ell)$ of $B_\ell$ in $A_\ell$, for $B_\ell$ the subalgebra of coinvariants of $A_\ell$ with respect to the $H$-coaction. Moreover $f_\ell: {}^{\co H_0}H \to  Z_{B_\ell} (A_\ell)$ are $H$-module, comodule algebra maps
and they satisfy (\ref{glue}), that is
\beq\label{comm-d}
\rho_{\ell,\ell m} \circ f_\ell=\rho_{m,\ell m} \circ f_m
\eeq
for  $\rho_{\ell, \ell m }$ the restriction morphisms  in \eqref{restr-maps}.
\end{proposition}
\begin{proof}

By using the commutation relations in \eqref{comm-rel-nu}, we see that 
the elements $\gn{\ell}_{1s}$ commute with the generators $ \gb{\ell}_r= \T_{\ell 1}^{-1} \T_{r1}$ in 
  \eqref{coinv-gln} of the subalgebra of coinvariants $B_\ell$:
$$
\gn{\ell}_{1 s} \gb{\ell}_r = \gn{\ell}_{1 s} \T_{\ell 1}^{-1} \T_{r1} 
= \mu_{s1} \T_{\ell 1}^{-1} \gn{\ell}_{1 s} \T_{r1}
 =\mu_{s1}  \mu_{1s} \T_{\ell 1}^{-1} \T_{r1}  \gn{\ell}_{1 s} =  \gb{\ell}_r \gn{\ell}_{1 s}
$$  
  where we use property \eqref{prop2}. 
Thus, the image of  $f_\ell$ is contained in the centralizer $Z_{B_\ell} (A_\ell)$.

Next we show that $f_\ell$ is an $H$-module and comodule  map. Being an algebra map, it is enough to show the properties on the algebra generators.
The $H$-comodule map condition \eqref{com-cod} is easily verified: using the expression for the coaction $\delta_\ell: A_\ell \to A_\ell \otimes H$ in \eqref{coact-nu}, we compute
$$
\delta_\ell(f_\ell(\beta_s)) = \delta_{\ell} (\gn{\ell}_{1s}) = \sum_{r=2}^n \gn{\ell}_{1r} \otimes p_{11}^{-1} p_{rs} =\sum_{r=2}^n f_\ell(\beta_r)\otimes p_{11}^{-1} p_{rs}  = f_\ell(\one{(\beta_s)}) \otimes \two{(\beta_s)}
$$
where  we used
$\Delta(\beta_s)=\Delta(p_{11}^{-1}p_{1s})= \sum_{r=2}^n  p_{11}^{-1} p_{1r} \otimes p_{11}^{-1}p_{rs} = \sum_{r=2}^n  \beta_r   \otimes p_{11}^{-1}p_{rs} $
for the last equality.

In order to show that the map $f_\ell$ is also an $H$-module map, that is \eqref{MU-condition} is satisfied, 
$$
f_\ell \big( S(\one{h}) \beta_s \two{h} \big)= \tone{h} f_\ell(\beta_s) \ttwo{h} \; , \quad
\forall h \in H, s=2, \dots, n, 
$$
we need the explicit expression of the translation map of the $H$--Galois extension $B_\ell \subset A_\ell$ that was given in \eqref{mappa-can-inv}.
For $h=p_{ir}$, $r>1$,  a generator of $H$, we have
\begin{align*}
f_\ell \big( \sum_{k=2}^n S(p_{i k}) \beta_s p_{k r} \big)
= f_\ell \big(\mu_{1r}\mu_{rs} \sum_{k=2}^n S(p_{i k}) p_{k r} \beta_s ) \big)
= \mu_{1r}\mu_{rs} \delta_{ir} f_\ell (  \beta_s  )
\end{align*}
which coincides with
$$
\tone{p_{ir}} f_\ell(\beta_s) ~ \ttwo{p_{ir}} 
= \sum_{k=1}^n S(\T_{i k}) f_\ell(\beta_s) \T_{k r}  
= \mu_{1r}\mu_{rs}  S(\T_{i k})   \T_{k r}    f_\ell (  \beta_s  )
= \mu_{1r}\mu_{rs} \delta_{ir} f_\ell (  \beta_s  ) \, .
$$
For $h=p_{11}$  we have
\begin{align*}
f_\ell \big( S(p_{11}) \beta_s p_{11}) \big)
= f_\ell \big(\mu_{1s}   S(p_{11}) p_{11} \beta_s ) \big)
= \mu_{1s}  f_\ell (  \beta_s  )
\end{align*}
equal to 
$$
\tone{p_{11}} f_\ell(\beta_s) ~ \ttwo{p_{11}} 
= \T_{\ell 1}^{- 1} f_\ell(\beta_s) \T_{\ell 1}
=\mu_{1s} \T_{\ell 1}^{- 1} \T_{\ell 1}  f_\ell(\beta_s)
=\mu_{1s}   f_\ell(\beta_s)
$$
and similarly, condition \eqref{MU-condition} is shown to hold for $h=\tilde{D}^{-1}$
by using that $\beta_s \tilde{D}^{-1} = q^{-2(s-1)} \tilde{D}^{-1}  \beta_s$
as well as $\gn{\ell}_{1s}  {D}^{-1} = q^{-2(s-1)} {D}^{-1}  \gn{\ell}_{1s}$.
It is clear from the proof, that the above identities hold valid if we replace 
$\beta_s$ by a generic element $ k \in {}^{\co H_0}H$, being this just a polynomial on the $\beta_s$.
Thanks to Lemma \ref{lem:MUprod} we conclude that \eqref{MU-condition} holds in general, that is the map $f_\ell$ is an $H$-module map as claimed.

\medskip
Now we go to the proof of (\ref{comm-d}).
Using the restriction morphisms $\rho_{\ell, \ell m }$ in \eqref{restr-maps}, $\ell<m$,
$$
\rho_{\ell, \ell m } \big( f_\ell (\beta_s) \big) 
=\rho_{\ell, \ell m } (\gn{\ell}_{1s}) =\gn{\ell}_{1s}
= \Psi_{m  \ell} (\gn{m}_{1s})
=\rho_{m, \ell m}   (\gn{m}_{1s})
= \rho_{m, \ell m} \big( f_m (\beta_s) \big)
$$
where we used the isomorphisms  $\Psi_{m  \ell}$ of Proposition \ref{prop:iso} in \eqref{psi-lm2}.
\end{proof}

\begin{corollary} Let the notation be as above. Then $\cF$ admits
a reduction to an $H_0$ quantum principal bundle $\cF^0$. 
\end{corollary}
By  Proposition \ref{prop:gunther-sheaf},  on the opens   of the covering one has
$\cF_0(U_\ell)= A_\ell \slash \cI(U_\ell)$
where $\cI(U_\ell)= \cF(U_\ell)  f_\ell \big( {}^{\co H_0}H \cap \ker \varepsilon_H \big)$. 
Since $\varepsilon_H(p_{1 s})=0$, we have here ${}^{\co H_0}H \cap \ker \varepsilon_H= {}^{\co H_0}H$. Moreover, being $f_\ell$ an algebra map, its 
image is the subalgebra of $A_\ell$ generated by the elements $\gn{\ell}_{1s}$. Then  
$$
\cF_0(U_\ell)= A_\ell \slash \cI(U_\ell) = A_\ell^0
$$
for $A_\ell^0$ the $H_0$-comodule algebras in \eqref{F0-GLn} defining the reduction sheaf $\cF_0$ of  \S \ref{sec:qrF}.

\appendix
\section{Proof of Proposition \ref{prop:Floc-cleft}}\label{app:A}

In this Appendix, we conclude the proof of Proposition \ref{prop:Floc-cleft}  showing that the map 
$$
\gamma_2: \widetilde{\cO}_q(P)  \to A_2 \; , \qquad  
\tilde{D}^p t^m   n^{r} \mapsto  (-1)^p q^p  D^p c^{m} d^r  
$$
defined in \eqref{cleaving2-gl2} is indeed a cleaving map for the algebra extension $B_2 \subset A_2$, which   is thus cleft as claimed.
First we show $\gamma_2$
is an $\widetilde{\cO}_q(P)$-comodule map: $(\gamma_2 \otimes id)\circ \Delta = \delta_2 \circ \gamma_2$. 
For that we first compute
$$
\Delta(\tilde{D}^p)= \tilde{D}^p \otimes \tilde{D}^p \; , \quad
\Delta(t^m)= t^m \otimes t^m \; , \quad
\Delta(n^r) = \sum_{j=0}^r \binom{r}{j} q^{j(r-j)} t^j n^{r-j} \otimes \tilde{D}^{r-j} t^{j-r}n^j  \; ,
$$
where the last formula is proven by induction and the symbol $\binom{r}{j}$ denotes the (classical) binomial coefficient. Moreover, for $m,p \in   \mathbb{Z}$ and $ r \in   \mathbb{N}$,  we compute
$$
\delta_2({D}^p)= {D}^p \otimes \tilde{D}^p \; , \quad
\delta_2(c^m)= c^m \otimes t^m \; , \quad
\delta_2(d^r) = \sum_{j=0}^r \binom{r}{j} q^{j(r-j)} c^j d^{r-j} \otimes \tilde{D}^{r-j} t^{j-r}n^j  \; ,
$$
using again induction. Then,  
\begin{align*}
(\gamma_2 \otimes id)\circ \Delta (\tilde{D}^p t^m  n^{r} )
& = \sum_{j=0}^r \binom{r}{j} q^{j(r-j)} \gamma_2 \big(
\tilde{D}^p t^{m+j} n^{r-j} 
\big)
\otimes \tilde{D}^{p+r-j} t^{m+j-r} n^j 
\\
& =  (-q)^p  \sum_{j=0}^r \binom{r}{j} q^{j(r-j)}    
{D}^p c^{m+j} d^{r-j} 
\otimes \tilde{D}^{p+r-j} t^{m+j-r} n^j  
\\
&=  (-q)^p \delta_2 ({D}^p c^m  d^{r} ) 
=
\delta_2 \circ \gamma_2 (\tilde{D}^p t^m  n^{r} )  \; ,
\end{align*}
showing that $\gamma_2$ is a comodule map, as claimed.
Next we show that $\gamma_2$
is  convolution invertible. Its inverse $\bar{\gamma}_2: \widetilde{\cO}_q(P)  \to A_2$   is
given by  
\beq\label{cleaving-gl2-inv2}
\bar{\gamma}_2 (\tilde{D}^p t^m   n^{r} ) := 
(-1)^p q^{-p  + r (r-1)   }   d^r c^{-m}     {D}^{-r-p}   \, .
\eeq
It is indeed immediate to verify that
\begin{align*}
(\gamma_2 * \bar{\gamma}_2)(\tilde{D}^p t^m    ) &=
 \gamma_2 \big(\tilde{D}^p t^{m}  \big)  \bar{\gamma}_2 (\tilde{D}^{p} t^{m} )
 =   
 {D}^p c^{m}    c^{-m}  {D}^{-p}   =1 = \varepsilon (\tilde{D}^p t^m    )
\end{align*}
and similarly 
$( \bar{\gamma}_2 *\gamma_2)(\tilde{D}^p t^m    )= 1$.
While for $r\neq 0$,  
\begin{align*}
(\gamma_2 * \bar{\gamma}_2)(\tilde{D}^p t^m  n^{r} ) &=
 \sum_{j=0}^r \binom{r}{j} q^{j(r-j)} \gamma_2 \big(
\tilde{D}^p t^{m+j} n^{r-j} 
\big) \bar{\gamma}_2 (\tilde{D}^{p+r-j} t^{m+j-r} n^j )
\\
& = \sum_{j=0}^r  \binom{r}{j}  (-1)^{r-j} q^{j(r-j)} q^p q^{-p-r+j +j(j-1)}    
{D}^p c^{m+j} d^{r-j}  d^j c^{-m-j+r}  {D}^{-p-r} 
  \\
  & = (-1)^r \sum_{j=0}^r  \binom{r}{j}  (-1)^{j} q^{jr -r }    
q^{-r(m+j-2p)}  d^{r}    c^{ r}  {D}^{-r} 
  \\
  & = (-1)^r   q^{-r (m-2p+1)}     d^{r}    c^{ r}  {D}^{-r} \sum_{j=0}^r  \binom{r}{j}(-1)^{j}    =0
\end{align*}
together with
\begin{align*}
( \bar{\gamma}_2 &* \gamma_2 )(\tilde{D}^p t^m  n^{r} ) \\&=
 \sum_{j=0}^r \binom{r}{j} q^{j(r-j)} \bar{\gamma}_2 \big(
\tilde{D}^p t^{m+j} n^{r-j} 
\big)
 {\gamma}_2 (\tilde{D}^{p+r-j} t^{m+j-r} n^j )
\\
& = \sum_{j=0}^r \binom{r}{j}  (-1)^{r-j} q^{j(r-j)}   
   q^{-p+(r-j)(r-j-1)}  q^{p+r-j}
d^{r-j}c^{-m-j}    {D}^{-p-r+j} {D}^{p+r-j} c^{m+j-r} d^j
\\
& =  (-1)^{r} \sum_{j=0}^r \binom{r}{j} (-1)^{j}  q^{r(r-j)}  
q^{-r(r-j)} c^{-r}   d^r
  \\
& =  (-1)^r c^{-r}  d^r   \sum_{j=0}^r  \binom{r}{j} (-1)^j =0.
\end{align*}
This concludes the proof that 
$\gamma_2$
  in \eqref{cleaving2-gl2} is   a cleaving map for the extension.

\end{document}